\documentclass[11pt]{article}
\usepackage[in]{fullpage}
\usepackage[utf8]{inputenc}
\usepackage{graphicx}
\usepackage{hyperref}
\usepackage{cite}
\usepackage{amsmath,amssymb,amsthm}
\newtheorem{theorem}{Theorem}
\newtheorem{lemma}[theorem]{Lemma}
\newtheorem{proposition}[theorem]{Proposition}
\newtheorem{corollary}[theorem]{Corollary}
\newtheorem{claim}{Claim}

\usepackage{lineno}

\newcommand{\GG}{\mathcal{G}}
\newcommand{\defeq}{:=}

\newcommand*\samethanks[1][\value{footnote}]{\footnotemark[#1]}

\title{Monotone edge flips to an orientation of maximum edge-connectivity \`a la Nash-Williams}
\author{Takehiro Ito\thanks{Graduate School of Information Sciences, Tohoku University, Japan} \and 
        Yuni Iwamasa\thanks{Graduate School of Informatics, Kyoto University, Japan} \and
        Naonori Kakimura\thanks{Faculty of Science and Technology, Keio University, Japan} \and
        Naoyuki Kamiyama\thanks{Institute of Mathematics for Industry, Kyushu University, Japan} \and
        Yusuke Kobayashi\thanks{Research Institute for Mathematical Sciences, Kyoto University, Japan} \and
        Shun-ichi Maezawa\thanks{Graduate School of Informatics and Engineering, The University of Electro-Communications, Japan} \and
        Yuta Nozaki\thanks{Graduate School of Advanced Science and Engineering, Hiroshima University, Japan} \and
        Yoshio Okamoto\samethanks[6] \and
        Kenta Ozeki\thanks{Faculty of Environment and Information Sciences, Yokohama National University, Japan}
        }
\date{\today}

\begin{document}

\maketitle

\begin{abstract}
    We initiate the study of $k$-edge-connected orientations of undirected graphs through edge flips for $k \geq 2$.
    We prove that in every orientation of an undirected $2k$-edge-connected graph, there exists a sequence of edges
    such that flipping their directions one by one does not decrease the edge-connectivity, and the final orientation is $k$-edge-connected.
    This yields an ``edge-flip based'' new proof of Nash-Williams' theorem: an undirected graph $G$ has a $k$-edge-connected orientation if and only if $G$ is $2k$-edge-connected.
    As another consequence of the theorem, we prove that the edge-flip graph of $k$-edge-connected orientations of an undirected graph $G$ is connected if $G$ is $(2k+2)$-edge-connected. 
    This has been known to be true only when $k=1$.
\end{abstract}

\section{Introduction}

An orientation of undirected graphs has been a subject of
thorough studies over several decades.
For an undirected graph $G=(V,E)$ with possible multiple edges, an \emph{orientation} of $G$ is a directed graph
$D=(V,A)$ obtained from $G$ by choosing a directed edge $(u,v) \in A$ or $(v,u) \in A$ for each undirected edge $\{u,v\} \in E$.

An old result by Robbins~\cite{robbins} states that an undirected graph $G$ has a strongly connected orientation if and only if $G$ is $2$-edge-connected.
Robbins' theorem was extended by Nash-Williams~\cite{nash-williams_1960} as an undirected graph $G$ has a $k$-edge-connected orientation if and only if $G$ is $2k$-edge-connected.

This paper is concerned with reorientation. 
A basic question asks to find a smallest set $F$ of edges in an orientation of a $2$-edge-connected graph such that flipping the directions of all edges in $F$ yields a strongly connected orientation.
By a theorem of Lucchesi and Younger~\cite{LucchesiYounger},
the problem can be solved in polynomial time.
A higher edge-connectedness version has also been studied, which asks to find a smallest set $F$ of edges in an orientation of a $2k$-edge-connected graph such that flipping the directions of all edges in $F$ yields a $k$-edge-connected orientation.
By submodular flow, the problem can be solved in polynomial time~\cite{FRANK198297}.
For a faster algorithm, see Iwata and Kobayashi~\cite{IwataK10}.

We now want to investigate the situation where flips are performed one by one sequentially, while the results in the literature studied flipping the edges of a set at once.
We also want each of the intermediate orientations in the process to maintain at least as high edge-connectivity as the previous orientations in the process.
This has practical importance since simultaneous edge flips can be difficult to implement or control in some real-world situations such as traffic management~\cite{HausknechtASFW11}, and the reduction of  edge-connectivity in intermediate orientations may cause the loss of network quality.

To make the discussion more precise, we define an \emph {edge flip} (or a \emph{flip} for short) of a directed edge $(u,v)$ as an operation that replaces $(u,v)$ by $(v,u)$, i.e., 
reverses the direction of $(u,v)$.
For directed graphs $D$ and $D'$,
we denote $D \rightarrow D'$
if $D'$ is obtained from $D$ by a single edge flip.

Our main theorem is the following.
Remind that the \emph{edge-connectivity} of a directed graph $D=(V,A)$ is the minimum integer $\lambda$ such that every non-empty subset $X\subsetneq V$ has at least $\lambda$ edges leaving $X$, and is denoted by $\lambda(D)$.

\begin{theorem}
\label{thm:main01}
Let $k$ be a non-negative integer. 
Let $G=(V, E)$ be an undirected $2k$-edge-connected graph and 
$D=(V,A)$ be an orientation of $G$.
Then, there exist orientations 
$D_1, D_2, \dots  , D_\ell$ of $G$ such that $\ell \le k|V|^3$, 
$D \rightarrow D_1 \rightarrow D_2 \rightarrow \dots \rightarrow D_\ell$, and 
$\lambda(D) \le \lambda(D_1) \le \lambda(D_2) \le \dots \le \lambda(D_\ell) = k$. 
Furthermore, such orientations $D_1, \dots , D_\ell$ can be found in polynomial time. 
\end{theorem}

Theorem \ref{thm:main01} states that for any orientation of a $2k$-edge-connected undirected graph $G$, there exists a sequence of edge flips such that the orientations of $G$ obtained by the successive edge flips have non-decreasing edge-connectivity and the resulting orientation is $k$-edge-connected.
Figure~\ref{fig:kconn_main1} shows an example.

\begin{figure}[t]
\centering
\includegraphics[width=\textwidth]{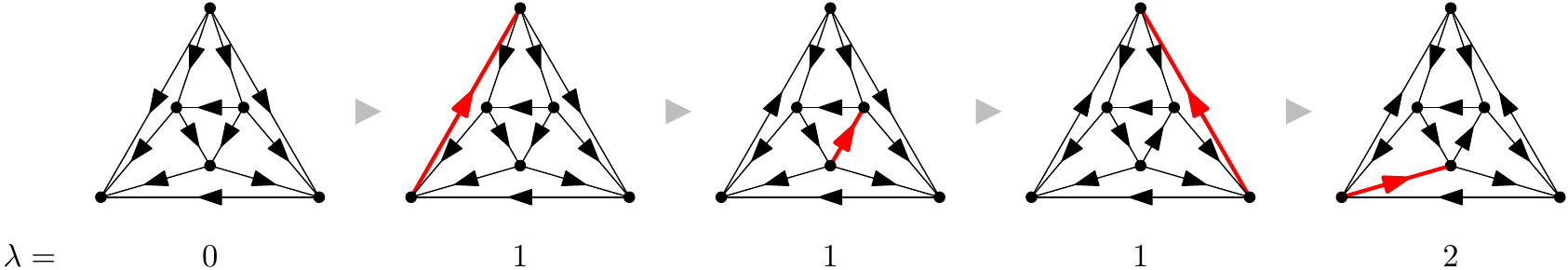}
\caption{An example for Theorem \ref{thm:main01}.
The underlying undirected graph is $4$-edge-connected. The number under each orientation shows its edge-connectivity. Fat red edges depict flipped edges. The list of orientations show that a sequence of edge flips leads to an orientation of edge-connectivity two while all the intermediate orientations have edge-connectivity one.}
\label{fig:kconn_main1}
\end{figure}

Theorem \ref{thm:main01} has several implications.
First, it provides another (algorithmic) proof of Nash-Williams' theorem \cite{nash-williams_1960} by edge flips.
It should be emphasized that the edge-connectivity of the orientation does not decrease during the transformation in Theorem \ref{thm:main01}, 
while Nash-Williams' theorem itself does not provide any guarantee for the edge-connectivity of intermediate orientations.

The second implication is concerned with the connectedness of the edge-flip graph of $k$-edge-connected orientations.
For an undirected graph $G=(V,E)$,
we define the \emph{edge-flip graph} $\GG_k(G)$ as the vertex set
is all the $k$-edge-connected orientations of $G$,
and two orientations are joined by an edge in the edge-flip graph if and only if one is obtained from the other by a single edge flip.
Figure \ref{fig:efgraph1_K4} is an example of the edge-flip graph of the strongly connected orientations.

\begin{figure}[ht]
    \centering
    \includegraphics[scale=0.4]{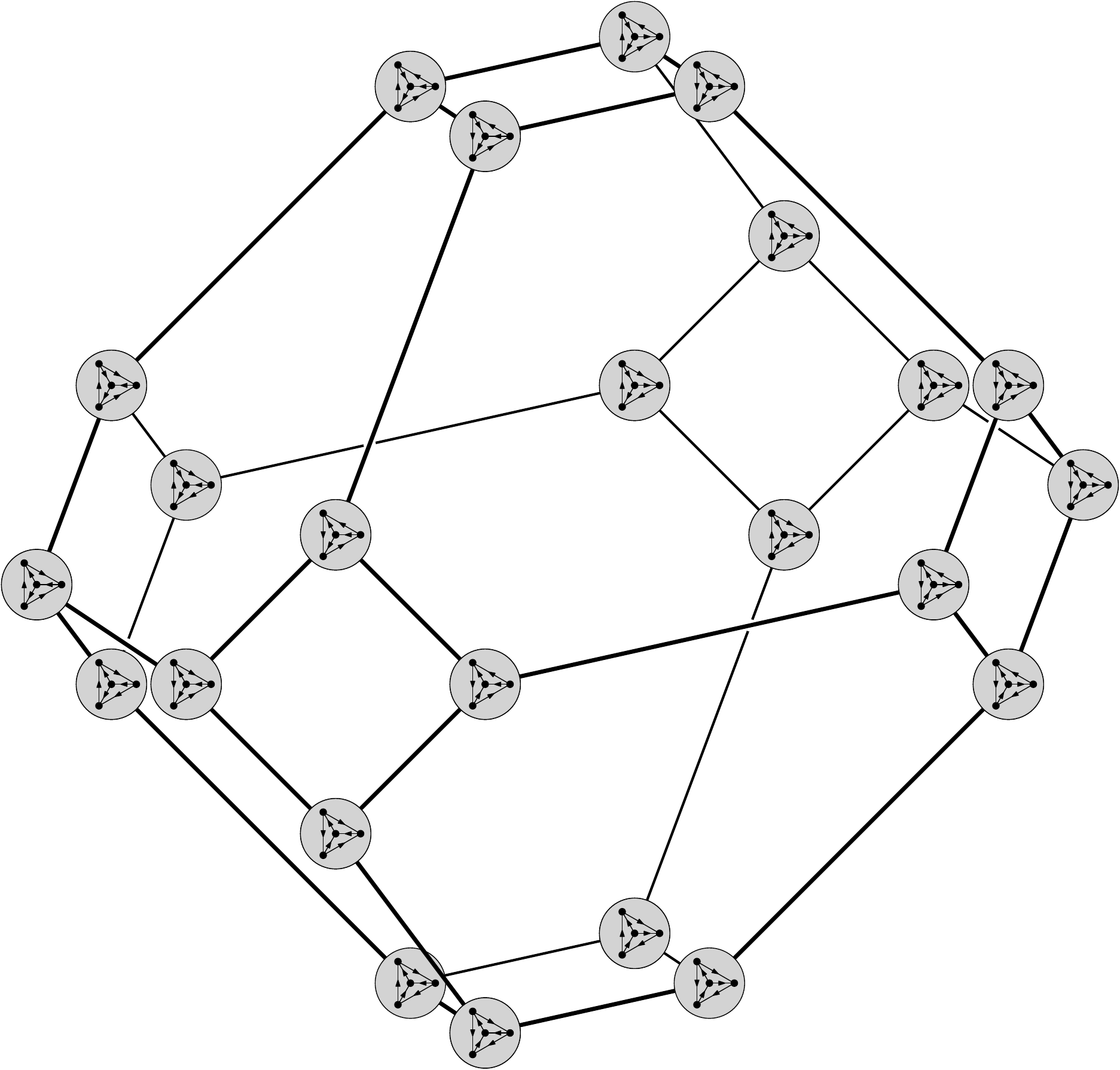}
    \caption{The edge-flip graph $\GG_1(K_4)$ of the strongly connected (i.e., $1$-edge-connected) orientations of a four-vertex complete graph $K_4$.}
    \label{fig:efgraph1_K4}
\end{figure}

Then, we consider the following two questions.
\begin{description}
\item[Global Reachability:] Given a connected undirected graph $G$, is the edge-flip graph $\GG_k(G)$ connected?
\item[Local Reachability:] Given a connected undirected graph $G$ and two $k$-edge-connected orientations $D_1, D_2$ of $G$,
is there a path connecting $D_1$ and $D_2$ in the edge-flip graph $\GG_k(G)$?
\end{description}

When $k=1$, the Global Reachability question is completely answered.
Greene and Zaslavsky~\cite{GZ} proved by hyperplane arrangements that the edge-flip graph $\GG_1(G)$ is connected if and only if $G$ is $3$-edge-connected.
Fukuda, Prodon, and Sakuma~\cite{FPS} gave a graph-theoretic proof for the same fact.

Our second result is a partial answer to the Global Reachability question for $k \ge 2$.
This is a higher-edge-connectedness analogue of the
result by Greene and Zaslavsky~\cite{GZ} and Fukuda, Prodon, and Sakuma~\cite{FPS}.
\begin{theorem}
\label{thm:main}
Let $k \geq 1$.
If $G$ is $(2k+2)$-edge-connected, then the edge-flip graph
$\GG_k(G)$ is connected.
\end{theorem}

Theorem \ref{thm:main} is obtained as a corollary of Theorem \ref{thm:main01}, combined with a result by Frank \cite{Frank82}.
The proof implies that the diameter of $\GG_k(G)$ is $O(k|V|^3+|E|^2)$
when $G$ is $(2k+2)$-edge-connected.
Note that $\GG_k(G)$ has no vertex if the edge-connectivity of $G$ is less than $2k$.   

We do not know if the $(2k+2)$-edge-connectedness can be replaced with the $(2k+1)$-edge-connectedness when $k \ge 2$. 
However, we know that we cannot replace it with the $2k$-edge-connectedness. 
Indeed, there exists a $2k$-edge-connected graph $G$ such that $\GG_k(G)$ is disconnected even when $k=1$ (e.g.\ consider the clockwise orientation and the counterclockwise orientation of a $3$-cycle).

For the Local Reachability question, we have the following characterization when $k=1$.
\begin{theorem}
\label{thm:local}
Let $G=(V,E)$ be a $2$-edge-connected graph and
$D_1, D_2$ be strongly connected orientations of $G$.
Then, there exists a path connecting $D_1$ and $D_2$
in the edge-flip graph $\GG_1(G)$ if and only if there exists
no $2$-edge-cut $\{\{u,v\}, \{u',v'\}\}$ such that
$(u,v), (v',u')$ are edges of $D_1$ and
$(v,u), (u',v')$ are edges of $D_2$.
Furthermore, a shortest path between two strongly connected orientations can be found in polynomial time if one exists.
\end{theorem}

We note that an analogous statement for higher edge-connectedness does not hold.
See the example in \figurename~\ref{fig:discon_example1}.
In this example, $D_1$ and $D_2$ are $2$-edge-connected orientations of a $4$-edge-connected graph $G$.
For any $4$-edge-cut in $G$, 
their direction in $D_1$ is the same as in $D_2$. 
On the other hand, one can see that there is no edge flip of $D_1$ that maintains $2$-edge-connectedness, 
which implies that $\GG_k(G)$ contains no path connecting $D_1$ and $D_2$. 

\begin{figure}[t]
\centering
\includegraphics{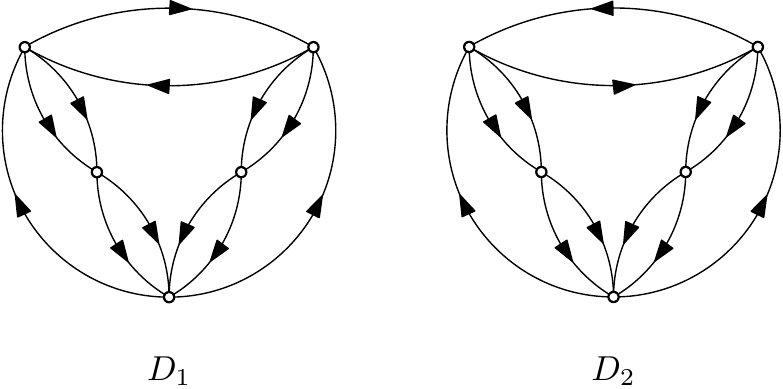}
\caption{This example shows that an analogous statement to Theorem \ref{thm:local} does not hold for $k=2$.}
\label{fig:discon_example1}
\end{figure}

\subsection*{Related Work}

Orientations of graphs have been a subject of intensive studies
in the literature of graph theory and combinatorial optimization.

Robbins~\cite{robbins} shows that an undirected graph $G$ has a strongly connected orientation if and only if $G$ is $2$-edge-connected.
Robbins' original proof is based on ear decompositions, which yields a linear-time algorithm~\cite{HopcroftT73}.
The Global Reachability of the edge-flip graph of strongly connected orientations is investigated by Greene and Zaslavsky \cite{GZ} and by 
Fukuda et al.~\cite{FPS}: 
They proved that the edge-flip graph of the strongly connected orientations of an undirected graph $G$ is connected if and only if $G$ is $3$-edge-connected, and in this case, a shortest path between two strongly connected orientations can be found in polynomial time.

Nash-Williams~\cite{nash-williams_1960} shows that an undirected graph $G$ has a $k$-edge-connected orientation if and only if $G$ is $2k$-edge-connected, where $k \geq 1$ is an integer.
Robbins' theorem \cite{robbins} corresponds to the case where $k=1$.
Nash-Williams' original proof is based on the so-called ``odd-node pairing theorem'' and Eulerian orientations.
See also~\cite{KiralyS06} for a simpler proof.
Since an odd-node pairing with the desired property can be found in polynomial time~\cite{Gabow94,NagamochiI97}, this proof technique yields a polynomial-time algorithm to find a $k$-edge-connected orientation of a $2k$-edge-connected graph.
Other proofs are based on the ``splitting-off theorem'' by Lov\'asz~\cite{lovasz_exercises} and submodular flows~\cite{frank_1993}.
Those two proofs also yield polynomial-time algorithms.

Nash-Williams' theorem can be generalized to the existence of an orientation satisfying a certain local connectivity constraint, 
which is called a well-balanced orientation~\cite{nash-williams_1960}. 
A further extension is shown by Fukunaga~\cite{FUKUNAGA20122349}. 
Bern\'{a}th et al.~\cite{BERNATH2008663} show some results on well-balanced orientations.  

When $k\geq 2$, 
the edge-flip graph of $k$-edge-connected orientations has not been studied.
Frank~\cite{Frank82} proved that the path/cycle-flip graph of $k$-edge-connected orientations of an undirected graph $G$ is connected if and only if $G$ is $2k$-edge-connected, where a \emph{path/cycle flip} is an operation that flips all the edges of a directed path or a directed cycle simultaneously.
Since this result will be used in our proof of Theorem \ref{thm:main}, we highlight it in the following theorem, where we also include a bound for the length of a sequence that was implicit in his proof.
\begin{theorem}[Frank \cite{Frank82}]
\label{thm:Frank82}
Let $k\geq 1$ be an integer, $G=(V,E)$ be a $2k$-edge-connected undirected graph, and $D_1, D_2$ be two $k$-edge-connected orientations of $G$.
Then, $D_1$ and $D_2$ can be transformed with each other by a sequence of path/cycle flips in such a way that all the intermediate orientations are $k$-edge-connected.
The length of such a sequence is bounded by $O(|E|)$ from above, and can be found in polynomial time.
\end{theorem}

Acyclic orientations are well-studied objects.
An orientation is \emph{acyclic} if it has no directed cycle.
It is easy to see that every undirected graph has an acyclic orientation.
The Global Reachability question is completely answered.
Greene and Zaslavsky~\cite{GZ} gave a geometric proof that
the edge-flip graph of acyclic orientations is connected.
Fukuda, Prodon, and Sakuma~\cite{FPS} gave a graph-theoretic proof for the same fact.
This trivially answers the Local Reachability question, too.
Indeed, their proofs give a shortest path between 
two acyclic orientations in the edge-flip graph, which can be found in polynomial time.

Degree-constrained orientations form another class of well-studied orientations.
In this case, we are also given a non-negative number $m(v)$ for every vertex $v$ of an undirected graph $G = (V,E)$.
Hakimi \cite{hakimi} proved that 
there exists an orientation of $G$ such that every vertex $v$ has the in-degree of $m(v)$
if and only if $|E|=\sum_{v \in V}m(v)$ and 
$|\{e \in E \mid e \subseteq X\}| \leq \sum_{v \in X}m(v)$ for all $X \subseteq V$: such an orientation can be found in polynomial time if exists.

To define the flip graph of degree-constrained orientations, an edge flip is useless since
a single edge flip does not keep the required degree property of the orientations.
Instead, we consider a \emph{cycle flip} that flips all the edges in a single directed cycle simultaneously.
A cycle flip preserves the property that the in-degree of every vertex $v$ is $m(v)$.
Therefore, the \emph{cycle-flip graph} of degree-constrained orientations has been studied in the literature.

The Global Reachability of the cycle-flip graph of degree-constrained orientations is known to hold as long as it is non-empty (see~\cite{frank_book}).
Thus, the Local Reachability question is again trivial.
However, computing a shortest path in the cycle-flip graph of degree-constrained orientations is NP-hard \cite{AichholzerCHKMS21,ItoKK0O19}.

Orientations with vertex-connectivity constraints are also studied in the literature. 
It is conjectured by Thomassen~\cite{THOMASSEN1989402} that, 
for any positive integer $k$, there exists a smallest positive integer $f(k)$ such that every $f(k)$-connected graph has a $k$-connected orientation.
Frank~\cite{10.5555/233157.233167} proposed a stronger conjecture: for any positive integer $k$, 
a graph $G=(V, E)$ has a $k$-connected orientation 
if and only if $G - U$ is $2(k-|U|)$-edge-connected for any $U \subseteq V$ with $|U| \le k$. 
Jord\'{a}n~\cite{JORDAN2005257} shows that 
$f(2) \le 18$ based on a result by Berg and Jord\'{a}n~\cite{DBLP:journals/jgt/BergJ06}, 
and this upper bound is improved to $14$ by Cheriyan, Durand de Gevigney, and Szigeti~\cite{CHERIYAN201417}. 
Thomassen~\cite{THOMASSEN201567} proves Frank's conjecture for $k=2$, that is, 
a graph $G=(V, E)$ admits a $2$-connected orientation 
if and only if it is $4$-edge-connected and $G-v$ is $2$-edge-connected for every $v \in V$. 
This implies that $f(2)=4$. 
For general $k$, Frank's conjecture was disproved recently by Durand de Gevigney~\cite{DURANDDEGEVIGNEY2020105}.
The existence of $f(k)$ is still open for $k \ge 3$. 

Frank, Kir\'{a}ly, and Kir\'{a}ly~\cite{FRANK2003385} proved that 
many known graph orientation theorems can be extended to hypergraphs.

\subsection*{Organization}
The rest of this paper is organized as follows. 
In Section~\ref{sec:preliminaries}, we give preliminary definitions and useful facts. 
In Section~\ref{sec:reconfproof}, 
we prove Theorem~\ref{thm:main} by using Theorem~\ref{thm:main01}. 
A proof of Theorem~\ref{thm:main01}, which is the main result in this paper, 
is given in Section~\ref{sec:mainproof}. 
In Section~\ref{sec:localproof}, we give a proof of Theorem~\ref{thm:local}. 
In Section~\ref{sec:conclusion}, we conclude this paper with some remarks.


\section{Preliminaries}
\label{sec:preliminaries}

\subsection{Undirected Graphs}

An \emph{undirected graph} $G=(V,E)$ is a pair of
its vertex set $V$ and its edge set $E$, where
each edge $e \in E$ is specified by an unordered pair $\{u,v\}$ of vertices: in this case
$u$ and $v$ are \emph{endpoints} of the edge $e$.
We allow multiple edges, and thus $E$ is considered
a multiset.

A \emph{path} in an undirected graph $G=(V,E)$
is a sequence $v_0,v_1,\dots,v_\ell$ of distinct vertices
such that $\{v_i, v_{i+1}\} \in E$ for every
$i\in \{0,1,\dots,\ell-1\}$: in this case
the path \emph{connects} $v_0$ and $v_\ell$, and 
$\ell$ is the \emph{length} of the path.
A path that connects $v_0$ and $v_\ell$ is also called
a \emph{$(v_0,v_\ell)$-path}.

For an undirected graph $G=(V,E)$ and a vertex subset
$S \subseteq V$, we denote by $E_G(S)$ the set of edges of $G$ that
have one endpoint in $S$ and the other endpoint in $V  -  S$:
\[
    E_G(S) \defeq \{e \in E \mid |e \cap S| = 1\}.
\]
Define $\delta_G(S) \defeq |E_G(S)|$.
The set $E_G(S)$ is a \emph{$k$-edge-cut} of $G$ if $|E_G(S)|=k$.

For $k\geq 1$, 
an undirected graph $G=(V,E)$ is \emph{$k$-edge-connected} 
if $\delta_G(S) \geq k$ for every non-empty $S \subsetneq V$.
By Menger's theorem, this is equivalent to the condition that 
there exists a set of $k$ edge-disjoint $(s,t)$-paths for every
pair of distinct vertices $s,t\in V$.
An undirected graph is \emph{connected} if it is
$1$-edge-connected: 
an undirected graph is \emph{disconnected} if it is 
not connected.
It is easy to observe that 
$G$ is $k$-edge-connected if and only if
there exists no $\ell$-edge-cut in $G$ for $\ell = 0,1,\dots,k-1$ except $E_G(\emptyset)$ and $E_G(V)$.

\subsection{Directed Graphs}

A \emph{directed graph} $D=(V,A)$ is a pair of
its vertex set $V$ and its edge set $A$, where
each directed edge $e \in A$ is specified by an ordered pair
$(u,v)$ of vertices: in this case
$u$ is the \emph{tail} of $e$ and $v$ is the \emph{head}
of $e$.
We allow multiple edges, and thus $A$ is considered a
multiset.

A \emph{path} in a directed graph $D=(V,A)$
is a sequence $v_0,v_1,\dots,v_\ell$ of distinct vertices
such that $(v_i, v_{i+1}) \in A$ for every
$i\in \{0,1,\dots,\ell-1\}$: in this case
the path \emph{connects} $v_0$ and $v_\ell$, and 
$\ell$ is the \emph{length} of the path.
A path that connects $v_0$ and $v_\ell$ is also called
a \emph{$(v_0,v_\ell)$-path}.
For $i, j \in \{0,1,\dots,\ell-1\}$ with $i \le j$, 
let $P[v_i, v_j]$ denote the subpath of $P$ from $v_i$ to $v_j$, that is, 
$P[v_i, v_j]$ is the sequence $v_i, v_{i+1}, \dots , v_j$. 
For a path $P$, the set of vertices (resp.~edges) in $P$ is denoted by $V(P)$ (resp.~$A(P)$). 

Let $D=(V,A)$ be a directed graph and $S \subseteq V$ a vertex subset.
The subgraph induced by $S$ is denoted by $D[S]$.
An edge $e \in A$ \emph{leaves} $S$
if the tail of $e$ belongs to $S$ but the head of $e$ does not belong to $S$.
Similarly, an edge $e \in A$ \emph{enters} $S$
if the head of $e$ belongs to $S$ but the tail of $e$ does not belong to $S$.
We denote by $\Delta^+_D(S)$ the set of edges in $A$ that leave $S$, and similarly 
by $\Delta^-_D(S)$ the set of edges in $A$ that enter $S$:
\[
    \Delta^+_D(S) \defeq \{ (u,v) \in A \mid u \in S, v \not\in S\},
    \qquad
    \Delta^-_D(S) \defeq \{ (u,v) \in A \mid u \not\in S, v \in S\}.
\]
Define $\delta^+_D(S) \defeq |\Delta^+_D(S)|$ and
$\delta^-_D(S) \defeq |\Delta^-_D(S)|$.

For $k\geq 1$, 
a directed graph $D=(V,A)$ is \emph{$k$-edge-connected} 
if $\delta^+_D(S) \geq k$
and $\delta^-_D(S) \geq k$ 
for every non-empty $S \subsetneq V$.\footnote{%
One of the conditions is actually redundant since $\delta^-_D(S)\geq k$ implies $\delta^+_D(V -  S)\geq k$.
}
By Menger's theorem, this is equivalent to the condition that
there exists a set of $k$ edge-disjoint $(s,t)$-paths for every
pair of distinct vertices $s,t\in V$.
A directed graph is \emph{strongly connected}
if it is $1$-edge-connected.
The \emph{edge-connectivity} of a directed graph $D$ is the
maximum integer $k$ such that $D$ is $k$-edge-connected, and is denoted by $\lambda(D)$.

Simple counting shows that the functions $\delta^+_D, \delta^-_D$ satisfy the following inequalities:
for all  $X, Y \subseteq V$,
\[
    \delta^+_D(X) + \delta^+_D(Y) \geq \delta^+_D(X\cap Y) + \delta^+_D(X \cup Y),\qquad 
    \delta^-_D(X) + \delta^-_D(Y) \geq \delta^-_D(X\cap Y) + \delta^-_D(X \cup Y).
\]
The two inequalities are referred to as \emph{submodularity}.

\section{Proof of Theorem~\ref{thm:main}}
\label{sec:reconfproof}

Let $k \geq 1$ and $G$ be an undirected $(2k+2)$-edge-connected graph.
We will show that the edge-flip graph
$\GG_k(G)$ is connected by using Theorem \ref{thm:main01}.

Let $D_1,D_2$ be $k$-edge-connected orientations of $G$.
We want to find a sequence of edge flips that transforms $D_1$ to $D_2$ in such a way that all the intermediate orientations are $k$-edge-connected.

Below is our strategy.
\begin{enumerate}
    \item We apply Theorem \ref{thm:main01} to transform $D_1$ to a $(k+1)$-edge-connected orientation $D'_1$ by edge flips so that all the intermediate orientations are $k$-edge-connected.
    This can be done by 
    the assumption that $G$ is $(2k+2)$-edge-connected.
    We apply the same procedure to $D_2$ to obtain a $(k+1)$-edge-connected orientation $D'_2$.
    \item Then, we apply Theorem \ref{thm:Frank82} due to Frank~\cite{Frank82} to transform $D'_1$ to $D'_2$. Since operations in Theorem \ref{thm:Frank82} are path/cycle flips, we need to turn them into sequences of edge flips.
    We emphasize that all the intermediate orientations will be $k$-edge-connected, but not necessarily $(k+1)$-edge-connected.
    \item Finally, we consider the reverse sequence of edge flips that transformed $D_2$ to $D'_2$ from the first step. Combining them, we obtain a sequence of edge flips that transforms $D_1$ to $D_2$ such that all the intermediate orientations are $k$-edge-connected.
\end{enumerate}

In the strategy above, the first and the third steps are clear, and the numbers of necessary steps are $O(k|V|^3)$ by Theorem \ref{thm:main01}.
We concentrate on the second step.
Let $D'_1$, $D'_2$ be $(k+1)$-edge-connected orientations of $G$.
Then, by Theorem \ref{thm:Frank82}, there exists a sequence of path/cycle flips that transforms $D'_1$ to $D'_2$ in such a way that all the intermediate orientations are $(k+1)$-edge-connected.
Let $D'_1=\hat{D}_0, \hat{D}_1, \dots, \hat{D}_\ell = D'_2$ be a sequence of orientations of $G$ such that $\hat{D}_{i}$ is obtained from $\hat{D}_{i-1}$ by a path/cycle flip and $\hat{D}_i$ is $(k+1)$-edge-connected for $i \in \{1,\dots,\ell\}$.

We now fix $i \in \{1,\dots,\ell\}$, and we will construct a sequence of orientations from $\hat{D}_{i-1}$ to $\hat{D}_{i}$ by edge flips.
Let $F$ be a directed path/cycle in $\hat{D}_{i-1}$ such that flipping the edges in $F$ yields $\hat{D}_{i}$.
Suppose that $F$ traverses arcs $e_1, e_2, \dots , e_m$ in this order, where $m$ is the number of edges in $F$.
Then, we flip $e_1, e_2, \dots , e_m$ one by one in this order.
The obtained sequence of orientations is denoted by $\hat{D}_{i-1} = \tilde{D}_0, \tilde{D}_1, \dots, \tilde{D}_m = \hat{D}_{i}$.
Note that $\tilde{D}_j$ is obtained from $\tilde{D}_0$ by flipping the edges of the path $\{e_1, \dots, e_j\}$.

We prove $\tilde{D}_j$ is $k$-edge-connected for any $j\in \{1,\dots,m-1\}$.
Let $X$ be a non-empty proper subset of $V$. 
We already know $\tilde{D}_0$ is $(k+1)$-edge-connected, that is, 
$\delta^+_{\tilde{D}_0}(X) \ge k+1$ and 
$\delta^-_{\tilde{D}_0}(X) \ge k+1$.
Furthermore, since $\{e_1,e_2,\dots,e_j\}$ forms a path in $\tilde{D}_0$,
\[
\Big|\big|\Delta^+_{\tilde{D}_0}(X) \cap \{e_1,e_2,\dots,e_j\}\big|
-
\big|\Delta^-_{\tilde{D}_0}(X) \cap \{e_1,e_2,\dots,e_j\}\big|\Big| 
\le 1. 
\]
This implies that $\delta^+_{\tilde{D}_j}(X) \ge k$ and 
$\delta^-_{\tilde{D}_j}(X) \ge k$.
Hence, $\tilde{D}_j$ is $k$-edge-connected for any $j \in \{0,\dots,m\}$.

Therefore, we can construct a sequence of $k$-edge-connected orientations from $D'_1$ to $D'_2$ by edge flips.
The length of the sequence for the second step is $O(|E|^2)$ by Theorem~\ref{thm:Frank82}.
This completes the proof. 
\qed


\section{Proof of the Main Theorem (Theorem \ref{thm:main01})}
\label{sec:mainproof}

Remind that for two orientations $D$ and $D'$ of $G$, we denote $D \rightarrow D'$ if $D'$ is obtained from $D$ by a single edge flip.
In this section, we show that, 
given a $k$-edge-connected orientation $D$ of $G$, 
we can increase the edge-connectivity of $D$ via a sequence of edge flips
without losing the $k$-edge-connectedness.

\begin{theorem}
\label{thm:main02}
Let $k$ be a non-negative integer. Let $G=(V, E)$ be an undirected $(2k+2)$-edge-connected graph and 
$D=(V,A)$ be a $k$-edge-connected orientation of $G$. 
Then, there exist orientations 
$D_1, D_2, \dots  , D_\ell$ of $G$ such that $\ell \le |V|^3$, 
$D \rightarrow D_1 \rightarrow D_2 \rightarrow \dots \rightarrow D_\ell$, 
$\lambda(D_i) \ge k$ for $i \in \{1, \dots , \ell-1\}$, and $\lambda(D_\ell) \ge k+1$. 
Furthermore, such $D_1, \dots , D_\ell$ can be found in polynomial time. 
\end{theorem}

Note that the $(2k+2)$-edge-connectedness is necessary for an undirected graph $G$ to have a $(k+1)$-edge-connected orientation.

Theorem \ref{thm:main01} is then a simple corollary
of Theorem \ref{thm:main02} as exhibited in the
proof below.

\begin{proof}[Proof of Theorem \ref{thm:main01}]
If $\lambda(D) = k$, then the claim is obvious. 
Otherwise, let $p :=\lambda(D) < k$ and $D^p := D$. 
Since $G$ is $(2p+2)$-edge-connected, by applying Theorem~\ref{thm:main02} with $D^p$, 
we obtain orientations 
$D^p_1, D^p_2, \dots  , D^p_{\ell_p}$ of $G$ such that $\ell_p \le |V|^3$, 
$D^p \rightarrow D^p_1 \rightarrow D^p_2 \rightarrow \dots \rightarrow D^p_{\ell_p}$, 
$\lambda(D^p_i) \ge p$ for $i \in \{1, \dots , \ell_p-1\}$, and $\lambda(D^p_{\ell_p}) \ge p+1$. 
By taking a subsequence if necessary, we may assume that 
$\lambda(D^p_i) = p$ for $i \in \{1, \dots , \ell_p-1\}$. 
Note that, since $D^p_{\ell_p}$ is obtained from $D^p_{\ell_p-1}$ by flipping exactly one edge,
we obtain $\lambda(D^p_{\ell_p}) - \lambda(D^p_{\ell_p-1}) \le 1$, and hence 
$\lambda(D^p_{\ell_p}) = p+1$. 
Then, set $D^{p+1} := D^p_{\ell_p}$ and apply Theorem~\ref{thm:main02} again with $D^{p+1}$ to obtain a sequence 
$D^{p+1}_1, D^{p+1}_2, \dots  , D^{p+1}_{\ell_{p+1}} =: D^{p+2}$. 
We repeat this procedure until the edge-connectivity becomes $k$. 
Then, 
\[
D^{p}_1, D^{p}_2, \dots  , D^{p}_{\ell_{p}}, 
D^{p+1}_1, D^{p+1}_2, \dots  , D^{p+1}_{\ell_{p+1}}, 
\dots, 
D^{k-1}_1, D^{k-1}_2, \dots  , D^{k-1}_{\ell_{k-1}}
\]
is a desired sequence, because 
$\lambda(D^i_1) = \dots = \lambda(D^i_{\ell_i-1}) = i$ and $\lambda(D^i_{\ell_i}) = i+1$ for $i \in \{p, p+1, \dots , k-1\}$, and 
the length of the sequence is $\sum_{i=p}^{k-1} \ell_i \le k |V|^3$.
Such a sequence can be found in polynomial time by Theorem~\ref{thm:main02}. 
\end{proof}

In the rest of this section, we give a proof of Theorem~\ref{thm:main02}. 

\subsection{Proof Outline}
\label{sec:outline}

We first consider the case of $k=0$. 
Since $G$ is $2$-edge-connected, it has a strongly connected orientation $D'=(V, A')$ (see \cite{robbins}). 
Let $r \in V$ be an arbitrary vertex. 
Then, by considering the union of an in-tree an out-tree both rooted at $r$, one can see that
there exists a subgraph $D''=(V, A'')$ of $D'$ such that $D''$ is strongly connected and $|A''| \le 2|V|-2$. 
Let $F \subseteq A''$ be the set of arcs whose directions are different in $D''$ and $D$. 
Then, by flipping edges in $F$ one by one in an arbitrary order, 
we obtain a sequence of orientations satisfying the conditions in Theorem~\ref{thm:main02}, 
because we have no constraint on the intermediate orientations when $k=0$, and 
the length of this sequence is at most $2|V|-2$.

In what follows in this section, we suppose that $k$ is a positive integer. 
Let $G=(V, E)$ be an undirected $(2k+2)$-edge-connected graph and 
$D=(V,A)$ be a $k$-edge-connected orientation of $G$.  
Throughout this section, we fix a vertex $r \in V$ arbitrarily. To simplify the notation, for $v \in V$, $\{v\}$ is sometimes denoted by $v$ if no confusion may arise. Define $\mathcal{F}_{\rm out}(D)$ and  $\mathcal{F}_{\rm in}(D)$ as 
\begin{linenomath}
\begin{align*}
    \mathcal{F}_{\rm out}(D) &:= \{ X \subseteq V - r \mid \delta^+_D(X) = k \} \cup \{V\}, \\ 
    \mathcal{F}_{\rm in}(D) &:= \{ X \subseteq V - r \mid \delta^-_D(X) = k \} \cup \{V\}. 
\end{align*}
\end{linenomath}
Throughout this paper, a set in $\mathcal{F}_{\rm out}(D)$ (resp.~$\mathcal{F}_{\rm in}(D)$) is shown by a blue (resp.~red) oval in the figures. 
Note that, for a vertex set $X \subseteq V$ with $r \in X$, 
$\delta^+_D(X) = k$ if and only if $V - X \in \mathcal{F}_{\rm in}(D)$.  
With this observation, we see that 
$D$ is $(k+1)$-edge-connected if and only if 
$\mathcal{F}_{\rm out}(D) = \mathcal{F}_{\rm in}(D) = \{V\}$. 
We also note that $\mathcal{F}_{\rm out}(D) \cap \mathcal{F}_{\rm in}(D) = \{V\}$, 
because $\delta^+_D(X) + \delta^-_D(X) = \delta_G(X) \ge 2k+2$ for any non-empty subset $X \subseteq V-r$.
Define $\mathcal{F}_{\rm min}(D)$ as the set of all inclusionwise minimal sets in $\mathcal{F}_{\rm out}(D) \cup \mathcal{F}_{\rm in}(D)$. 
As we will see in Corollary~\ref{cor:Fmin_disjoint}, 
$\mathcal{F}_{\rm min}(D)$ consists of disjoint sets. 
If $D$ is clear from the context, $\mathcal{F}_{\rm out}(D), \mathcal{F}_{\rm in}(D)$, and $\mathcal{F}_{\rm min}(D)$ are simply denoted by 
$\mathcal{F}_{\rm out}, \mathcal{F}_{\rm in}$, and $\mathcal{F}_{\rm min}$, respectively. 

In our proof of Theorem~\ref{thm:main02}, 
by flipping some edges in $D$,
we decrease the value of 
\[
{\sf val}(D) := \sum_{X \in \mathcal{F}_{\rm min}(D)} (|V| - |X|).
\]
We repeat this procedure as long as ${\sf val}(D)$ is positive.  
If this value becomes $0$, then $\mathcal{F}_{\rm min} = \{V\}$. 
This means that $\mathcal{F}_{\rm out} = \mathcal{F}_{\rm in} = \{V\}$, 
and hence $D$ is $(k+1)$-edge-connected. 
Note that we decrease the value of ${\sf val}(D)$ at most $|V|^2$ times, because
${\sf val}(D)$ is integral and ${\sf val}(D) \le |V|^2$. 
Therefore, to prove Theorem~\ref{thm:main02}, 
it suffices to show the following proposition. 

\begin{proposition}
\label{prop:dec_val}
Suppose that $\mathcal{F}_{\rm min} \not= \{V\}$. 
Then, there exist orientations 
$D_1, D_2, \dots  , D_\ell$ of $G$ such that $\ell \le |V|$, 
$D \rightarrow D_1 \rightarrow D_2 \rightarrow \dots \rightarrow D_\ell$, 
$\lambda(D_i) \ge k$ for $i \in \{1, \dots , \ell\}$, and 
${\sf val}(D_\ell) < {\sf val}(D)$. 
Furthermore, such $D_1, \dots , D_\ell$ can be found in polynomial time. 
\end{proposition}

In what follows in this section, we assume that $\mathcal{F}_{\rm min} \not= \{V\}$ 
and give an algorithm for finding such orientations as in Proposition~\ref{prop:dec_val}. 
In our algorithm, 
we find an inclusionwise minimal set $S$ in $\mathcal{F}_{\rm in}$, an inclusionwise minimal set $T$ in $\mathcal{F}_{\rm out}$, and a path $P$ from $S$ to $T$. 
Then, we flip the edges in $P$ one by one from one end to the other.
In order to obtain orientations with the conditions in Proposition~\ref{prop:dec_val}, 
we have to choose $P$ carefully. 
First, it is necessary that $P$ does not enter a set in $\mathcal{F}_{\rm in}$~(or does not leave a set in $\mathcal{F}_{\rm out}$), as otherwise flipping edges violates $k$-edge-connectivity.
Moreover, to decrease ${\sf val}(D)$, we choose a path $P$ so that it is from a {\it safe source} in $S$ to a {\it safe sink} in $T$~(see Section~\ref{sec:safevertex} for definitions).

We note that ${\sf val}(D)$, safe sources, and safe sinks are first introduced in this paper, 
and they are key ingredients in our arguments. 

After describing basic properties of $\mathcal{F}_{\rm out}, \mathcal{F}_{\rm in}$, and $\mathcal{F}_{\rm min}$ in Section~\ref{sec:basic}, 
we introduce safe sources and safe sinks in Section \ref{sec:safevertex}. 
Then, we describe our algorithm in Section~\ref{sec:mainproc}, and 
prove its validity in Section~\ref{sec:validity}. 
A proof of a key lemma in our algorithm is shown in Section~\ref{sec:defP}.

\subsection{Basic Properties}
\label{sec:basic}

In this subsection, we show some basic properties of $\mathcal{F}_{\rm out}, \mathcal{F}_{\rm in}$, and $\mathcal{F}_{\rm min}$.

\begin{lemma}
\label{lem:102}
For $X, Y \subseteq V$ with $X\cap Y \not= \emptyset$, 
we have the following. 
\begin{enumerate}
    \item If $X, Y \in \mathcal{F}_{\rm out}$, then $X \cap Y, X \cup Y \in \mathcal{F}_{\rm out}$. 
    \item If $X, Y \in \mathcal{F}_{\rm in}$, then $X \cap Y, X \cup Y \in \mathcal{F}_{\rm in}$.
\end{enumerate}
\end{lemma}

\begin{proof}
If $X = V$ or $Y=V$, then the claim is obvious. Otherwise, $X, Y \subseteq V-r$. 
Since $D$ is $k$-edge-connected, if $X, Y \in \mathcal{F}_{\rm out}$, then
\[
2 k = \delta^+_D(X) + \delta^+_D(Y) \ge \delta^+_D(X \cap Y) + \delta^+_D(X \cup Y) \ge 2k 
\]
by the submodularity of $\delta^+_D$. Here, we note that $X \cup Y \not= V$ since $r \not\in X \cup Y$. 
Therefore, 
$\delta^+_D(X \cap Y) = \delta^+_D(X \cup Y) = k$, which means that  $X \cap Y, X \cup Y \in \mathcal{F}_{\rm out}$. 
The same argument can be applied to $\mathcal{F}_{\rm in}$.  
\end{proof}

\begin{lemma}
\label{lem:101}
For $X, Y \subseteq V$, it holds that
$\delta^+_D(X) + \delta^-_D(Y) \ge \delta^+_D(X - Y) + \delta^-_D(Y - X)$.
\end{lemma}

\begin{proof}
Since $\delta^-_D(S) = \delta^+_D(V-S)$ for any $S \subseteq V$, we obtain 
\begin{linenomath}
\begin{align*}
    \delta^+_D(X) + \delta^-_D(Y) &= \delta^+_D(X) + \delta^+_D(V-Y) \\
                                  &\ge \delta^+_D( X \cap (V- Y) ) + \delta^+_D(X \cup (V - Y))  \\
                                  &= \delta^+_D( X - Y) + \delta^-_D(Y - X)
\end{align*}
\end{linenomath}
by the submodularity of $\delta^+_D$. 
\end{proof}

\begin{lemma}
\label{lem:103}
Suppose that $X \in \mathcal{F}_{\rm out}$, $Y \in \mathcal{F}_{\rm in}$, 
$X - Y \not= \emptyset$, and $Y - X \not= \emptyset$. 
Then, $X-Y \in \mathcal{F}_{\rm out}$ and $Y-X \in \mathcal{F}_{\rm in}$. 
\end{lemma}

\begin{proof}
Since $X - Y \not= \emptyset$ and $Y - X \not= \emptyset$, 
we have $X \not= V$ and $Y \not= V$, and hence $X, Y \subseteq V -r$. 
Since $D$ is $k$-edge-connected, we obtain
\[
2 k = \delta^+_D(X) + \delta^-_D(Y) \ge \delta^+_D(X - Y) + \delta^-_D(Y - X) \ge 2k 
\]
by Lemma~\ref{lem:101}. 
Therefore, 
$\delta^+_D(X - Y) = \delta^-_D(Y - X) = k$, which means that $X-Y \in \mathcal{F}_{\rm out}$ and $Y-X \in \mathcal{F}_{\rm in}$. 
\end{proof}

By these lemmas, we obtain the following corollary. 
Recall that $\mathcal{F}_{\rm min}$ is the family of all inclusionwise minimal sets in $\mathcal{F}_{\rm in} \cup \mathcal{F}_{\rm out}$.

\begin{corollary}
\label{cor:Fmin_disjoint}
$\mathcal{F}_{\rm min}$ consists of disjoint sets. 
\end{corollary}

\begin{proof}
Assume to the contrary that $\mathcal{F}_{\rm min}$ contains two distinct sets $X$ and $Y$ with $X \cap Y \not= \emptyset$. 
Then, Lemmas~\ref{lem:102} and~\ref{lem:103} show that $X \cap Y$ or $X-Y$ is in $\mathcal{F}_{\rm in} \cup \mathcal{F}_{\rm out}$, 
which contradicts the minimality of $X$. 
\end{proof}

We note that $\mathcal{F}_{\rm min}$ can be computed in polynomial time, 
because each inclusionwise minimal element of $\mathcal{F}_{\rm out}$ and $\mathcal{F}_{\rm in}$ 
can be computed by using a minimum cut algorithm.

\subsection{Safe Source and Safe Sink}
\label{sec:safevertex}

As described in Section~\ref{sec:outline}, 
we choose a path $P$ from a safe source to a safe sink in our algorithm. 
In this subsection, we introduce safe sources and safe sinks.

Let $S$ be an inclusionwise minimal vertex set in $\mathcal{F}_{\rm in}$ (or~$\mathcal{F}_{\rm out}$, respectively). 
A vertex $s \in S$ is called a {\em safe source in $S$} (resp.~a {\em safe sink in $S$}) if, for any $X \subseteq V - r$ with $s \in X$ and $S - X \not= \emptyset$, 
\begin{enumerate}
    \item $\delta^+_D (X) \ge k+1$ (resp.~$\delta^-_D (X) \ge k+1$) holds, and 
    \item if $\delta^+_D (X) = k+1$ (resp.~$\delta^-_D (X) = k+1$), then there exists a vertex set $X' \subseteq X -s$ with $X' \in \mathcal{F}_{\rm out}$ (resp.~$X' \in \mathcal{F}_{\rm in}$); see Figure~\ref{fig:03}.   
\end{enumerate}

\begin{figure}[t]
\centering
\includegraphics[width=5cm]{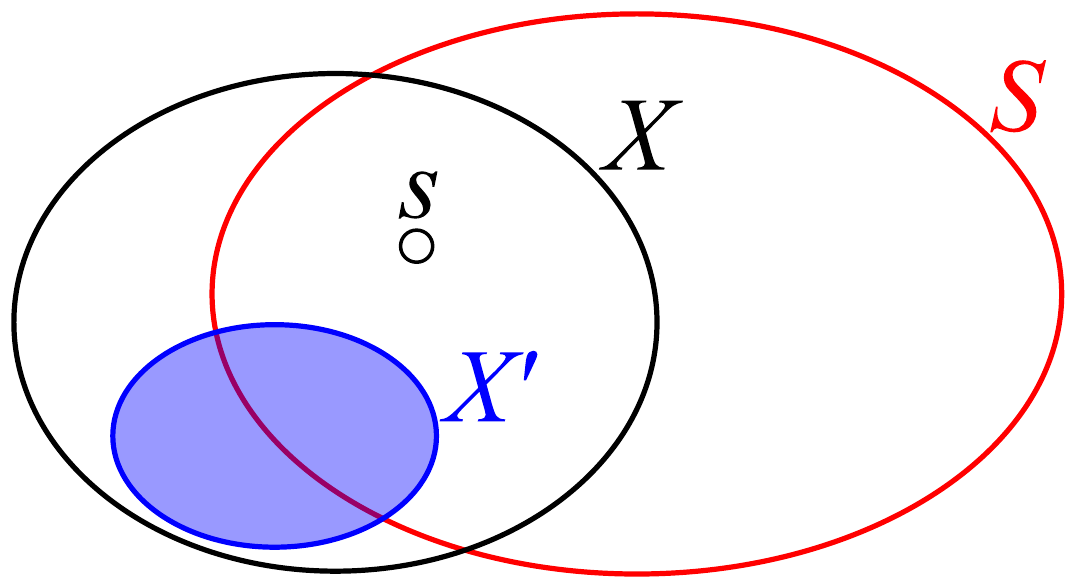}
\caption{The second condition.}
\label{fig:03}
\end{figure}

Here is an intuition of the definition.
In our algorithm~(see Section~\ref{sec:mainproc}),
we will find a path from a safe source $s$ in $S$ to a safe sink $t$ in $T$, and flip the edges of the path one by one from $t$ to $s$.
By the edge flips, $S$ and $T$ are removed from $\mathcal{F}_{\rm min}$, but new sets may be added to $\mathcal{F}_{\rm min}$. 
The definition of safety guarantees that a set $X$ newly becomes a member of $\mathcal{F}_{\rm min}$ only if $X\supseteq S$ or $X\supseteq T$.
For example, if a set $X$ with $s\in X$ and $t\not\in X$ satisfies $\delta^+_D (X)=k+1$, then $X$ may newly become a member of $\mathcal{F}_{\rm min}$.
However, the definition of a safe sink guarantees that $X$ has a proper subset $X'\in \mathcal{F}_{\rm out}$ contained in $X-s$, which implies that $X$ cannot become inclusionwise minimal after edge flips.
A similar argument holds for a safe sink. 
Therefore, a set $X$ newly becomes a member of $\mathcal{F}_{\rm min}$ only if $X\in \mathcal{F}_{\rm out}\cup \mathcal{F}_{\rm in}$.
This shows that ${\sf val}(D)$ decreases by at least one.
See the proof of Lemma~\ref{lem:decrease_val} for the details.

As we will see in Lemma~\ref{lem:107}, a safe source (resp.~a safe sink) always exists
in any inclusionwise minimal vertex set in $\mathcal{F}_{\rm in}$ (resp.~$\mathcal{F}_{\rm out}$).

\subsection{Our Algorithm}
\label{sec:mainproc}

In this subsection, we describe our algorithm for finding orientations $D_1, \dots , D_\ell$ with the conditions in Proposition~\ref{prop:dec_val}. 

Let $R \subseteq V$ be an inclusionwise minimal vertex set satisfying either 
\begin{description}
    \item[(a)] $R \in \mathcal{F}_{\rm in}$ and there exists a vertex set $X \subsetneq R$ with $X \in \mathcal{F}_{\rm out}$, or
    \item[(b)] $R \in \mathcal{F}_{\rm out}$ and there exists a vertex set $X \subsetneq R$ with $X \in \mathcal{F}_{\rm in}$.  
\end{description}
Note that such a vertex set $R$ always exists, 
since $\mathcal{F}_{\rm min}\neq \{V\}$ implies that $R = V$ satisfies (a) or (b). 
Furthermore, for each inclusionwise minimal set $X$ in $\mathcal{F}_{\rm out}$ (resp.~$\mathcal{F}_{\rm in}$), 
we can compute the unique minimal set $R'$ satisfying $R' \supsetneq X$ and $R' \in \mathcal{F}_{\rm in}$ (resp.~$R' \in \mathcal{F}_{\rm out}$) 
by a minimum cut algorithm, which shows that 
$R$ can be found in polynomial time. 
We also note that $X \subsetneq R$ in the conditions can be replaced with $X \subseteq R$ unless $R=V$, because $\mathcal{F}_{\rm in} \cap \mathcal{F}_{\rm out} = \{V\}$. 
By symmetry, we may assume that $R$ satisfies (a), i.e., 
$R \in \mathcal{F}_{\rm in}$ and there exists a vertex set $X \subsetneq R$ with $X \in \mathcal{F}_{\rm out}$. 
Define $\mathcal{F}^R_{\rm out}$ as 
\[
    \mathcal{F}^R_{\rm out} := \{ X \in \mathcal{F}_{\rm out} \mid X \subsetneq R \},  
\]
which is nonempty by the choice of $R$. 

We use the following key lemma, whose proof is given in Section~\ref{sec:defP}. 

\begin{lemma}\label{lem:defP}
Let $R \subseteq V$ be an inclusionwise minimal vertex set satisfying either (a) or (b). 
If $R$ satisfies (a), then 
we can find in polynomial time an $(s, t)$-path $P$ in $D[R]$ that consists of an $(s, t')$-path $Q_1$ and a $(t', t)$-path $Q_2$ for some $t' \in R$ satisfying the following (Figure~\ref{fig:15}).
\begin{enumerate}
    \item The vertex $s$ is a safe source in some inclusionwise minimal set $S\in \mathcal{F}_{\rm in}$ with $S\subseteq R$, and $t$ is a safe sink in some inclusionwise minimal set $T\in \mathcal{F}_{\rm out}$ with $T\subsetneq R$.
    \item $(V(Q_1) - t')\cap X = \emptyset$ for every $X\in  \mathcal{F}^R_{\rm out}$. That is, the subpath $Q_1$ is disjoint from any set in $\mathcal{F}^R_{\rm out}$, except for the end vertex $t'$.
    \item $A(Q_2)\cap \Delta^+_D(X) =\emptyset$ for every $X\in \mathcal{F}_{\rm out}$. That is, the subpath $Q_2$ does not intersect with $\Delta^+_D(X)$ for any $X\in \mathcal{F}_{\rm out}$.
\end{enumerate}
\end{lemma}

\begin{figure}[t]
\centering
\includegraphics[width=7cm]{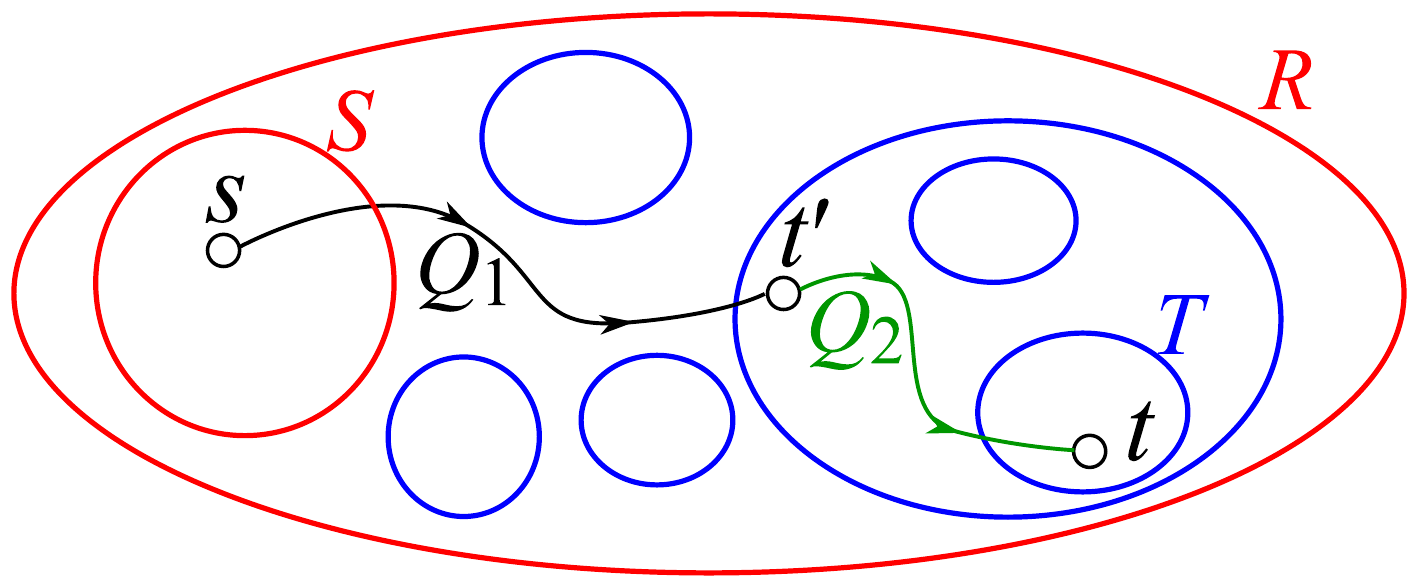}
\caption{Conditions in Lemma~\ref{lem:defP}.}
\label{fig:15}
\end{figure}

Let $P$ be a path satisfying the conditions in Lemma~\ref{lem:defP}. 
Suppose that $P$ traverses arcs $e_{\ell}, e_{\ell -1}, \dots , e_2$, and $e_1$ in this order from $s$ to $t$. 
Then, we flip $e_1, e_2, \dots , e_{\ell-1}$, and $e_\ell$ in this order. 
For $i=1, 2, \dots , \ell$, let $D_i$ be the directed graph obtained from $D$ by flipping 
$e_1, e_2, \dots , e_{i-1}$, and $e_i$. 
Then, our algorithm returns $D_1, D_2, \dots , D_\ell$.

We now show that $D_1, D_2, \dots  , D_{\ell-1}$, and $D_\ell$ satisfy the conditions in Proposition~\ref{prop:dec_val}. 
By Lemma~\ref{lem:defP}, $D_1, D_2, \dots  , D_{\ell-1}$, and $D_\ell$ can be computed in polynomial time.  
Furthermore, $\ell \le |V|$ and $D \rightarrow D_1 \rightarrow D_2 \rightarrow \dots \rightarrow D_\ell$ are obvious by definition. 
In the next subsection, we will prove $\lambda(D_i) \ge k$ for each $i \in \{1, 2, \dots , \ell\}$ and 
${\sf val}(D_\ell) < {\sf val}(D)$.


\subsection{Validity of the Algorithm}
\label{sec:validity}

Suppose that the algorithm described in the previous subsection finds a path $P$ satisfying the conditions in Lemma~\ref{lem:defP}, and returns $D_1, D_2, \dots  , D_{\ell-1}$, and $D_\ell$. 
The following two lemmas show that they satisfy the conditions in Proposition~\ref{prop:dec_val}.

\begin{lemma}
\label{lem:Di_is_kconnected}
For each $i \in \{1, 2, \dots , \ell\}$, $D_i$ is $k$-edge-connected. 
\end{lemma}

\begin{proof}
Assume to the contrary that $D_i$ is not $k$-edge-connected for some $i \in \{1, 2, \dots , \ell\}$, that is, 
$\delta^+_{D_i} (X) < k$ for some $X \subseteq V - r$ or $\delta^-_{D_i} (X) < k$ for some $X \subseteq V - r$. 
Let $p \in V$ be the tail of $e_i$. Then, since $D_i$ is obtained from $D$ by reversing the direction of the subpath of $P$ from $p$ to $t$,
we obtain 
\begin{linenomath}
\begin{align}
\delta^+_{D_i} (X) &= 
\begin{cases}
\delta^+_{D} (X) - 1 & \mbox{if } t \not\in X \mbox{ and } p \in X, \\
\delta^+_{D} (X) + 1 & \mbox{if } t \in X \mbox{ and } p \not\in X, \\
\delta^+_{D} (X)  &  \mbox{otherwise,}
\end{cases} \label{eq:01}\\
\delta^-_{D_i} (X) &= 
\begin{cases}
\delta^-_{D} (X) - 1 & \mbox{if } t \in X \mbox{ and } p \not\in X, \\
\delta^-_{D} (X) + 1 & \mbox{if } t \not\in X \mbox{ and } p \in X, \\
\delta^-_{D} (X)  &  \mbox{otherwise}
\end{cases} \label{eq:02}
\end{align}
\end{linenomath}
for any $X \subseteq V$. 
Since $D$ is $k$-edge-connected and $D_i$ is not $k$-edge-connected, there exists a vertex set $X^* \subseteq V-r$ such that either 
\begin{itemize}
    \item $\delta^+_{D} (X^*) = k$ (equivalently, $X^* \in \mathcal{F}_{\rm out}$), $t \not\in X^*$,  and $p \in X^*$, or 
    \item $\delta^-_{D} (X^*) = k$ (equivalently, $X^* \in \mathcal{F}_{\rm in}$), $t \in X^*$,  and $p \not\in X^*$. 
\end{itemize}
We treat the two cases separately. 
Recall that $t \in T$, where $T$ is a minimal vertex set in $\mathcal{F}_{\rm out}$. 

\bigskip
\noindent
\textbf{Case 1:} $X^* \in \mathcal{F}_{\rm out}$, $t \not\in X^*$, and $p \in X^*$ (see Figure~\ref{fig:08}). 

\begin{figure}[t]
\centering
\includegraphics[width=6cm]{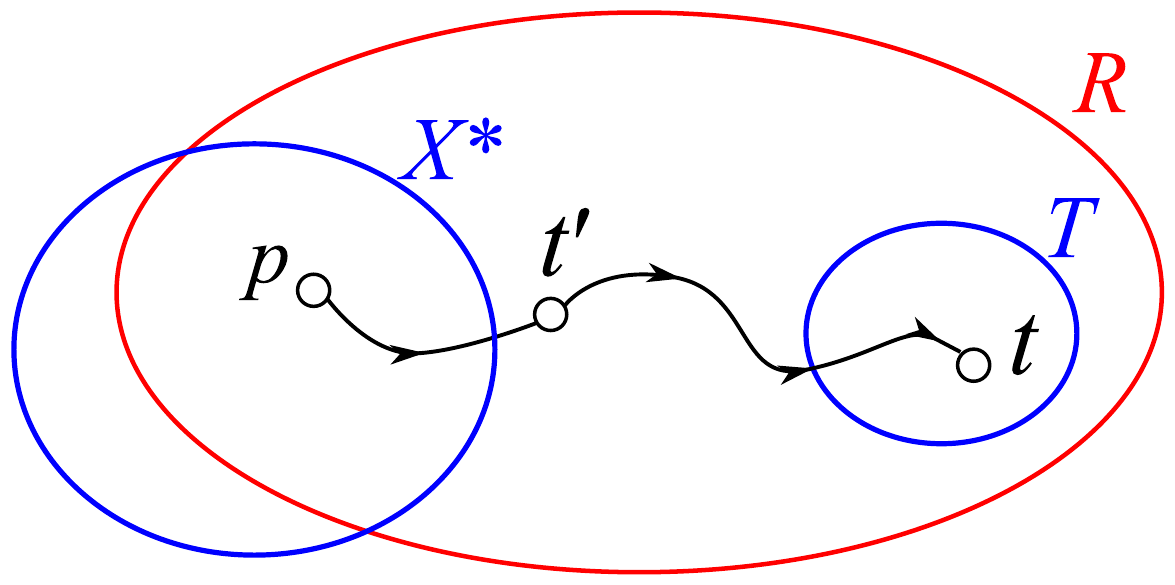}
\caption{Case 1 of Lemma~\ref{lem:Di_is_kconnected}.}
\label{fig:08}
\end{figure}

Recall that $P$ is the concatenation of $Q_1$ and $Q_2$ as in Lemma~\ref{lem:defP}. 
We can derive a contradiction if $p \in V(Q_2)$ or $X^* \cap T \not= \emptyset$ as follows. 
\begin{itemize}
    \item If $p \in V(Q_2)$, then $Q_2$ contains the $(p, t)$-path and hence it contains an edge in $\Delta^+_D(X^*)$, which contradicts that
        $A(Q_2) \cap \Delta^+_D(X) = \emptyset$ for any $X \in \mathcal{F}_{\rm out}$. 
    \item If $X^* \cap T \not= \emptyset$, then $X^* \cap T \in \mathcal{F}_{\rm out}$ by Lemma~\ref{lem:102}. 
        Since $X^* \cap T \subseteq T-t \subsetneq T$, this contradicts that $T$ is a minimal vertex set in $\mathcal{F}_{\rm out}$. 
\end{itemize}
Therefore, $p \in V(P) - V(Q_2) = V(Q_1) - t'$ and $X^* \cap T = \emptyset$ hold. 
By the second condition of Lemma~\ref{lem:defP}, we have $X^* \not\in \mathcal{F}^R_{\rm out}$. 
This together with $X^* \in \mathcal{F}_{\rm out}$ shows that $X^* - R \not= \emptyset$. 
We also see that $T \subseteq R - X^*$ holds by $X^* \cap T = \emptyset$, in particular, $R - X^* \not= \emptyset$ holds. 
Then, by applying Lemma~\ref{lem:103} to $R$ and $X^*$, we obtain $R - X^* \in \mathcal{F}_{\rm in}$. 
Since $T \subseteq R - X^*$ and $T \in \mathcal{F}_{\rm out}$, this shows that $R - X^*$ satisfies the condition (a). 
This contradicts the minimality of $R$, 
since $R-X^* \subseteq R-p \subsetneq R$. 

\bigskip
\noindent
\textbf{Case 2:} $X^* \in \mathcal{F}_{\rm in}$, $t \in X^*$, and $p \not\in X^*$. 

In this case, we derive a contradiction as follows. 
\begin{enumerate}
    \item[(i)] If $X^* \subseteq T$, then $T$ satisfies the condition (b), which contradicts the minimality of $R$ (Figure~\ref{fig:09}). 
    \item[(ii)] If $T \subseteq X^*$, then $R \cap X^* \in \mathcal{F}_{\rm in}$ by Lemma~\ref{lem:102}. 
    This together with $T \subseteq R \cap X^*$ shows that $R\cap X^*$ satisfies the condition (a). 
    Since $R \cap X^* \subseteq R - p \subsetneq R$, this contradicts the minimality of $R$ (Figure~\ref{fig:11}). 
    \item[(iii)] If $T-X^* \not= \emptyset$ and $X^*-T \not= \emptyset$, then $T-X^* \in \mathcal{F}_{\rm out}$ by Lemma~\ref{lem:103}. 
    Since $T-X^* \subseteq T-t \subsetneq T$, this contradicts that $T$ is a minimal vertex set in $\mathcal{F}_{\rm out}$ (Figure~\ref{fig:10}). 
\end{enumerate}

\begin{figure}[t]
 \begin{minipage}{0.33\hsize}
    \centering
    \includegraphics[width=3.8cm]{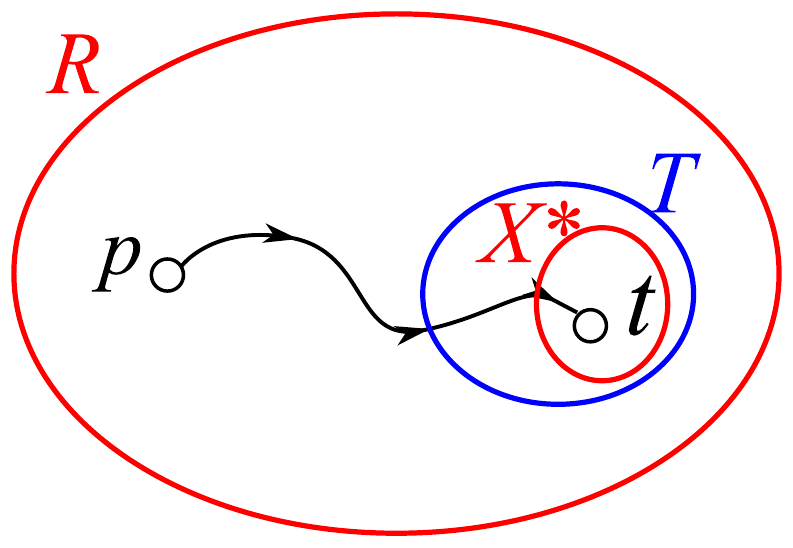}
    \caption{Case 2 (i).}
    \label{fig:09}
 \end{minipage}
 \begin{minipage}{0.33\hsize}
    \centering
    \includegraphics[width=4cm]{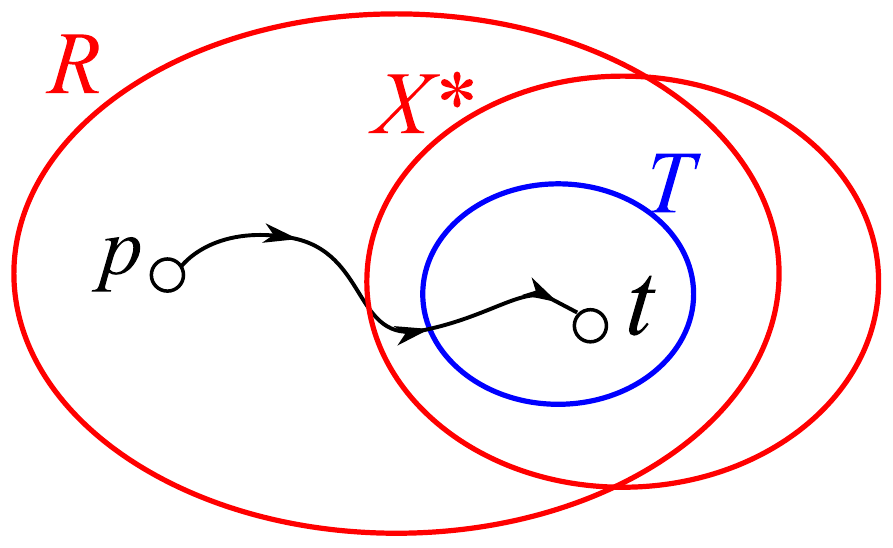}
    \caption{Case 2 (ii).}
    \label{fig:11}
 \end{minipage}
 \begin{minipage}{0.33\hsize}
    \centering
    \includegraphics[width=4cm]{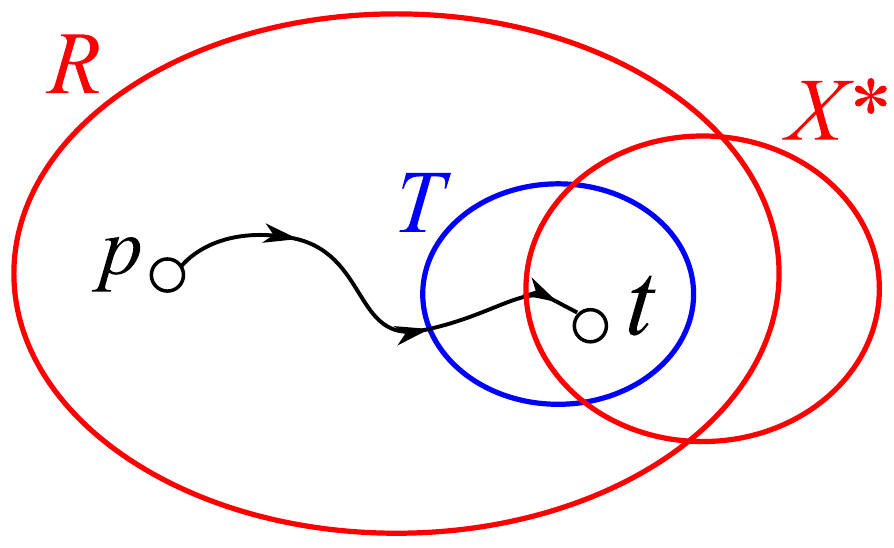}
    \caption{Case 2 (iii).}
    \label{fig:10}
 \end{minipage}
\end{figure}

By Cases 1 and 2, $D_i$ is $k$-edge-connected for each $i \in \{1, 2, \dots , \ell\}$. 
\end{proof}

We next show that ${\sf val}(D)$ is decreased by the procedure, where we recall that 
${\sf val}(D) :=  \sum_{X \in \mathcal{F}_{\rm min}(D)} (|V| - |X|)$.

\begin{lemma}
\label{lem:decrease_val}
${\sf val}(D_\ell) < {\sf val}(D)$. 
\end{lemma}

\begin{proof}
To simplify the notation, 
we denote $\mathcal{F}_{\rm out} := \mathcal{F}_{\rm out} (D)$, $\mathcal{F}_{\rm in} := \mathcal{F}_{\rm in} (D)$, $\mathcal{F}_{\rm min} := \mathcal{F}_{\rm min} (D)$,  
$\mathcal{F}_{\rm out}' := \mathcal{F}_{\rm out} (D_\ell)$, $\mathcal{F}_{\rm in}' := \mathcal{F}_{\rm in} (D_\ell)$, and $\mathcal{F}_{\rm min}' := \mathcal{F}_{\rm min} (D_\ell)$. 
Recall that $D_\ell$ is obtained from $D$ by reversing the direction of an $(s, t)$-path. 
In the same way as (\ref{eq:01}) and (\ref{eq:02}), we see that 
\begin{linenomath}
\begin{align}
\delta^+_{D_\ell} (X) &= 
\begin{cases}
\delta^+_{D} (X) - 1 & \mbox{if } t \not\in X \mbox{ and } s \in X, \\
\delta^+_{D} (X) + 1 & \mbox{if } t \in X \mbox{ and } s \not\in X, \\
\delta^+_{D} (X)  &  \mbox{otherwise,} \label{eq:03}
\end{cases} \\
\delta^-_{D_\ell} (X) &= 
\begin{cases}
\delta^-_{D} (X) - 1 & \mbox{if } t \in X \mbox{ and } s \not\in X, \\
\delta^-_{D} (X) + 1 & \mbox{if } t \not\in X \mbox{ and } s \in X, \\
\delta^-_{D} (X)  &  \mbox{otherwise} \label{eq:04}
\end{cases} 
\end{align}
\end{linenomath}
for any $X \subseteq V$. 
This shows that, to investigate the gap between $\mathcal{F}_{\rm min}$ and $\mathcal{F}_{\rm min}'$, 
it suffices to focus on sets containing $s$ or $t$. 
We treat the following two cases separately.

\bigskip
\noindent
\textbf{Case 1: $S = R$.} 

Recall that $t \in T$ and $T$ is an inclusionwise minimal vertex set in $\mathcal{F}_{\rm out}$. 
Since $S-T \not= \emptyset$, $\delta^+_D (T)=k$, and $s$ is a safe source in $S$, we obtain $s \not\in T$. 
The minimality of $R$ implies that $T$ does not contain a set in $\mathcal{F}_{\rm in}$, and hence $T \in \mathcal{F}_{\rm min}$. 
Since $\delta^+_{D_\ell} (T) = \delta^+_{D} (T) + 1 = k+1$ and $\delta^-_{D_\ell} (T) \ge (2k+2) - \delta^+_{D_\ell} (T) = k+1$, 
it holds that $T \in \mathcal{F}_{\rm min}  -  (\mathcal{F}_{\rm out}' \cup \mathcal{F}_{\rm in}')$. 

The following claim asserts that only the set $T$ is removed from $\mathcal{F}_{\rm min}$ and, if some set $X$ is newly added to $\mathcal{F}_{\rm min}'$, then $X\supsetneq T$ holds.

\begin{claim}
\label{clm:S=R}
If $S=R$, then it holds that $\mathcal{F}_{\rm min}' = \mathcal{F}_{\rm min}  -  \{T\}$ or 
$\mathcal{F}_{\rm min}' = (\mathcal{F}_{\rm min}  -  \{T\}) \cup \{X\}$ for some $X \supsetneq T$. 
\end{claim}
\begin{proof}[Proof of Claim \ref{clm:S=R}]
We first show that $\mathcal{F}_{\rm min}  -  (\mathcal{F}_{\rm out}' \cup \mathcal{F}_{\rm in}') = \{T\}$. 
Assume to the contrary that there exists a set 
$X \in \mathcal{F}_{\rm min}  -  (\mathcal{F}_{\rm out}' \cup \mathcal{F}_{\rm in}')$ with $X \not= T$. 
Since $X \in \mathcal{F}_{\rm min}  -  (\mathcal{F}_{\rm out}' \cup \mathcal{F}_{\rm in}') \subseteq (\mathcal{F}_{\rm out} \cup \mathcal{F}_{\rm in})  -  (\mathcal{F}_{\rm out}' \cup \mathcal{F}_{\rm in}')$,
by (\ref{eq:03}) and (\ref{eq:04}), 
it holds that $|X \cap \{s, t\}| = 1$.
This shows that $X \in \mathcal{F}_{\rm in}$, $t \not\in X$, and $s \in X$,  
because $T$ is the unique element in $\mathcal{F}_{\rm min}$ containing $t$.  
Since $s \in X \cap S$, $S \in \mathcal{F}_{\rm in}$, and $X \in \mathcal{F}_{\rm in} \cap \mathcal{F}_{\rm min}$, 
we obtain $X \subseteq S$, and hence $X \subseteq S-t \subsetneq S$. 
This contradicts the fact that $S$ is an inclusionwise minimal vertex set in $\mathcal{F}_{\rm in}$ with $s \in S$. 
Therefore, we obtain $\mathcal{F}_{\rm min}  -  (\mathcal{F}_{\rm out}' \cup \mathcal{F}_{\rm in}') = \{T\}$. 

We now claim that $X \supsetneq T$ holds for any 
$X \in \mathcal{F}_{\rm min}'  -  \mathcal{F}_{\rm min}$.
Let $X$ be a set in $\mathcal{F}_{\rm min}'  -  \mathcal{F}_{\rm min}$.
Since $X \not= T$ is obvious, it suffices to show $X \supseteq T$. 
Since $X \in \mathcal{F}_{\rm out}' \cup \mathcal{F}_{\rm in}'$, by (\ref{eq:03}) and (\ref{eq:04}), we have one of the following: 
(i) $X \in \mathcal{F}_{\rm out} \cup \mathcal{F}_{\rm in}$, 
(ii) $s \not\in X$, $t \in X$, and $\delta^-_{D} (X) = k+1$, or
(iii) $s \in X$, $t \not\in X$, and $\delta^+_{D} (X) = k+1$. 
Then, for each case, $X \supseteq T$ holds if such a set $X$ exists as follows. 
\begin{enumerate}
\item[(i)]
If $X \in \mathcal{F}_{\rm out} \cup \mathcal{F}_{\rm in}$, then $X \not\in \mathcal{F}_{\rm min}$
implies that there exists a set $Y \subsetneq X$ with $Y \in \mathcal{F}_{\rm min}$. 
Since $X \in \mathcal{F}_{\rm min}'$, it holds that
$Y \in \mathcal{F}_{\rm min}  -  (\mathcal{F}_{\rm out}' \cup \mathcal{F}_{\rm in}') = \{T\}$. 
Therefore, $Y$ must be equal to $T$, and hence $X \supsetneq T$. 

\item[(ii)]
Suppose that $s \not\in X$, $t \in X$, and $\delta^-_{D} (X) = k+1$ (see Figure~\ref{fig:12}). 
Assume to the contrary that $X \supseteq T$ does not hold, i.e., $T-X \not= \emptyset$. 
Since $t$ is a safe sink in $T$, there exists a vertex set $X' \subseteq X-t$ with $X' \in \mathcal{F}_{\rm in}$
by the definition of a safe sink. 
This shows that $X' \subsetneq X$ and $X' \in \mathcal{F}_{\rm in}'$ as $s, t \not\in X'$, 
which contradicts $X \in \mathcal{F}_{\rm min}'$. 
Therefore, $X \supseteq T$ holds. 

\item[(iii)]
Suppose that $s \in X$, $t \not\in X$, and $\delta^+_{D} (X) = k+1$ (see Figure~\ref{fig:13}).
Then, $R-X \not= \emptyset$, because it contains $t$. 
Since $s$ is a safe source in $R$, there exists a vertex set $X' \subseteq X-s$ with $X' \in \mathcal{F}_{\rm out}$
by the definition of a safe source. 
This shows that $X' \subsetneq X$ and $X' \in \mathcal{F}_{\rm out}'$ as $s, t \not\in X'$, 
which contradicts $X \in \mathcal{F}_{\rm min}'$. 
Therefore, such $X$ does not exist. 

\begin{figure}[t]
 \begin{minipage}{0.49\hsize}
    \centering
    \includegraphics[width=5.5cm]{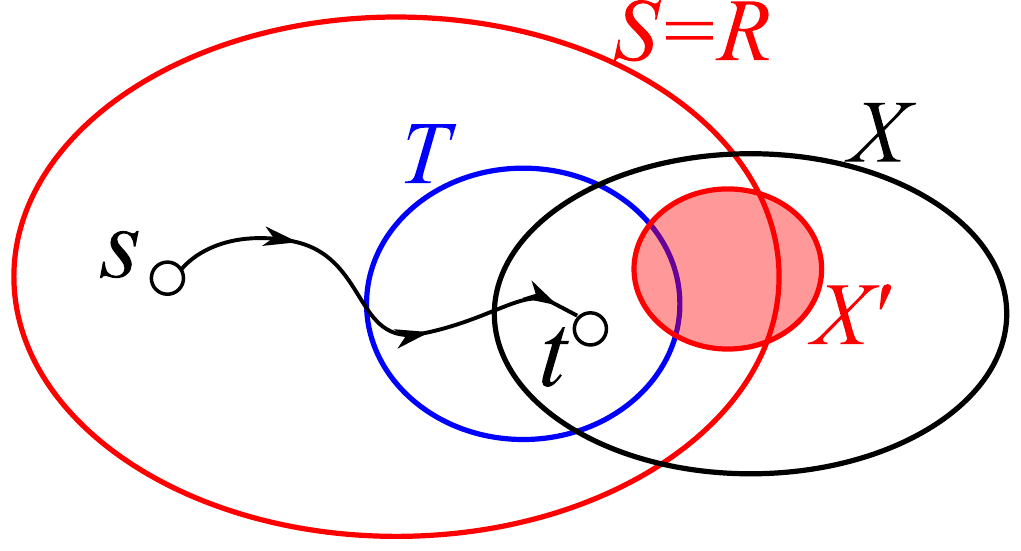}
    \caption{Case (ii) of Claim~\ref{clm:S=R}.}
    \label{fig:12}
 \end{minipage}
 \begin{minipage}{0.49\hsize}
    \centering
    \includegraphics[width=5.5cm]{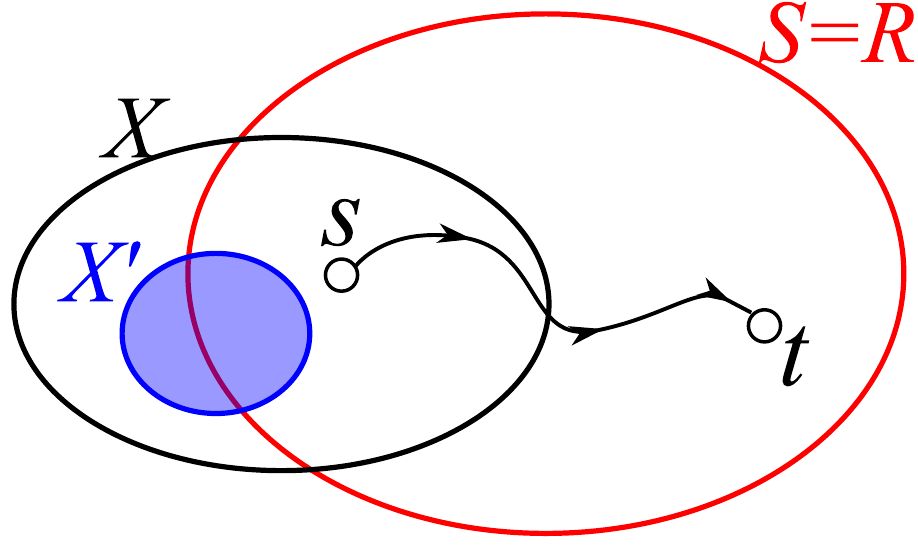}
    \caption{Case (iii) of Claim~\ref{clm:S=R}.}
    \label{fig:13}
 \end{minipage}
\end{figure}

\end{enumerate}

By the above argument, $X \supsetneq T$ holds for any 
$X \in \mathcal{F}_{\rm min}'  -  \mathcal{F}_{\rm min}$.
Since there exists at most one set in $\mathcal{F}_{\rm min}'$ containing $T$, 
we obtain $\mathcal{F}_{\rm min}' = \mathcal{F}_{\rm min}  -  \{T\}$ or 
$\mathcal{F}_{\rm min}' = (\mathcal{F}_{\rm min}  -  \{T\}) \cup \{X\}$ for some $X \supsetneq T$. 
Thus, Claim \ref{clm:S=R} holds.
\end{proof}

By Claim \ref{clm:S=R}, 
if $\mathcal{F}_{\rm min}' = \mathcal{F}_{\rm min}  -  \{T\}$, then ${\sf val}(D) - {\sf val}(D_\ell) = |V| - |T| > 0$, and, if $\mathcal{F}_{\rm min}' = (\mathcal{F}_{\rm min}  -  \{T\}) \cup \{X\}$ for some $X \supsetneq T$, then ${\sf val}(D) - {\sf val}(D_\ell) = |X| - |T| > 0$ holds.
This completes the proof for the case of $S = R$.

\bigskip
\noindent
\textbf{Case 2: $S \subsetneq R$.} 

Recall that $S$ is a minimal vertex set in $\mathcal{F}_{\rm in}$ and $s$ is a safe source in $S$. 
The minimality of $R$ implies that $S$ (resp.~$T$) does not contain a set in $\mathcal{F}_{\rm out}$ (resp.~$\mathcal{F}_{\rm in}$), and hence $S, T \in \mathcal{F}_{\rm min}$. This implies that $S$ and $T$ are disjoint. 
Since $\delta^-_{D_\ell} (S) = \delta^-_{D} (S) + 1 = k+1$, $\delta^+_{D_\ell} (S) \ge (2k+2) - \delta^-_{D} (S) = k+1$, 
$\delta^+_{D_\ell} (T) = \delta^+_{D} (T) + 1 = k+1$, and $\delta^-_{D_\ell} (T) \ge (2k+2) - \delta^+_{D_\ell} (T) = k+1$,
it holds that $S, T \in \mathcal{F}_{\rm min}  -  (\mathcal{F}_{\rm out}' \cup \mathcal{F}_{\rm in}')$. 
Therefore, by (\ref{eq:03}) and (\ref{eq:04}), we obtain $\mathcal{F}_{\rm min}  -  (\mathcal{F}_{\rm out}' \cup \mathcal{F}_{\rm in}') = \{S, T\}$. 
Moreover, we have the following claim.

\begin{claim}
\label{clm:Snot=R}
If $S\subsetneq R$, then one of the following cases holds:
\begin{itemize}
    \item $\mathcal{F}_{\rm min}' = \mathcal{F}_{\rm min}  -  \{S, T\}$, 
    \item $\mathcal{F}_{\rm min}' = (\mathcal{F}_{\rm min}  -  \{S, T\}) \cup \{X\}$ for some set $X$ with $X \supsetneq S$ or $X \supsetneq T$,
    \item $\mathcal{F}_{\rm min}' = (\mathcal{F}_{\rm min}  -  \{S, T\}) \cup \{X_s, X_t\}$ for some sets $X_s\supsetneq S$ and $X_t\supsetneq T$.
\end{itemize}
\end{claim}
\begin{proof}[Proof of Claim \ref{clm:Snot=R}]
We claim that $X \supsetneq S$ or $X \supsetneq T$ holds for any 
$X \in \mathcal{F}_{\rm min}'  -  \mathcal{F}_{\rm min}$.
Let $X$ be a set in $\mathcal{F}_{\rm min}'  -  \mathcal{F}_{\rm min}$.
Since $X \not= S, T$ is obvious, it suffices to show $X \supseteq S$ or $X \supseteq T$. 
Since $X \in \mathcal{F}_{\rm out}' \cup \mathcal{F}_{\rm in}'$, by (\ref{eq:03}) and (\ref{eq:04}), we have one of the following: 
(i) $X \in \mathcal{F}_{\rm out} \cup \mathcal{F}_{\rm in}$, 
(ii) $s \not\in X$, $t \in X$, and $\delta^-_{D} (X) = k+1$, or
(iii) $s \in X$, $t \not\in X$, and $\delta^+_{D} (X) = k+1$. 
For the cases (i) and (ii), we see that $X \supseteq S$ or $X \supseteq T$ holds in the same way as Case 1. 
For the case (iii), we can show that $X \supseteq S$ holds in the same way as (ii) as follows. 

Suppose that $s \in X$, $t \not\in X$, and $\delta^+_{D} (X) = k+1$ (Figure~\ref{fig:14}).
Assume to the contrary that $X \supseteq S$ does not hold, i.e., $S-X \not= \emptyset$. 
Since $s$ is a safe source in $S$, there exists a vertex set $X' \subseteq X-s$ with $X' \in \mathcal{F}_{\rm out}$
by the definition of a safe source. 
This shows that $X' \subsetneq X$ and $X' \in \mathcal{F}_{\rm out}'$ as $s, t \not\in X'$, 
which contradicts $X \in \mathcal{F}_{\rm min}'$. 
Therefore, $X \supseteq S$ holds. 

\begin{figure}[t]
\centering
\includegraphics[width=5.5cm]{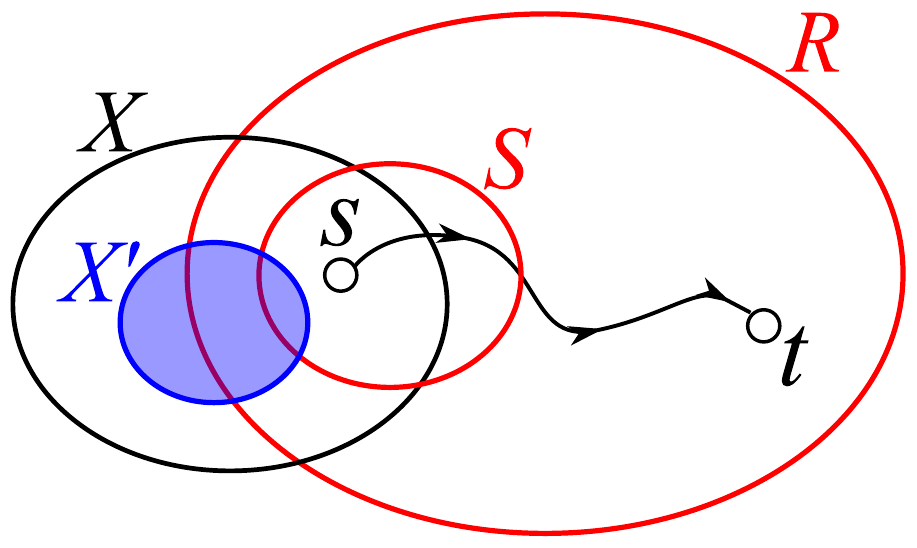}
\caption{Case (iii) of Claim~\ref{clm:Snot=R}.}
\label{fig:14}
\end{figure}

By the above argument, $X \supsetneq S$ or $X \supsetneq T$ holds for any 
$X \in \mathcal{F}_{\rm min}'  -  \mathcal{F}_{\rm min}$.
Note that there exists at most one set $X_s$ (resp.~$X_t$) in $\mathcal{F}_{\rm min}'$ containing $S$ (resp.~$T$). 
Thus the claim holds.
\end{proof}

It follows from Claim \ref{clm:Snot=R} that ${\sf val}(D) - {\sf val}(D_\ell)>0$.
Indeed, for the first case, ${\sf val}(D) - {\sf val}(D_\ell) = 2|V| - |S| - |T| > 0$;
for the second case, ${\sf val}(D) - {\sf val}(D_\ell) = |V| + |X| - |S| - |T| > 0$;
for the last case, ${\sf val}(D) - {\sf val}(D_\ell) = |X_s| + |X_t| - |S| - |T| > 0$.
This completes the proof for the case of $S \subsetneq R$, and closes the whole proof of Lemma \ref{lem:decrease_val}.
\end{proof}

Lemmas~\ref{lem:Di_is_kconnected} and~\ref{lem:decrease_val} show that the output of our algorithm in Section~\ref{sec:mainproc}
satisfies the conditions in Proposition~\ref{prop:dec_val}. 
By repeatedly applying Proposition~\ref{prop:dec_val} at most $|V|^2$ times, we obtain Theorem~\ref{thm:main02}. 
\qed

\subsection{Construction of a Path (Proof of Lemma~\ref{lem:defP})}
\label{sec:defP}

In this section, we first show some useful lemmas in Sections~\ref{sec:lem107proof} and~\ref{sec:existpath}, 
and then give a proof of Lemma~\ref{lem:defP} in Section~\ref{sec:subsectiondefP}.

\subsubsection{Existence of a Safe Source and a Safe Sink}
\label{sec:lem107proof}

\begin{lemma}
\label{lem:107}
For any inclusionwise minimal vertex set $S$ in $\mathcal{F}_{\rm in}$ (or~$\mathcal{F}_{\rm out}$, respectively), 
there exists a safe source (resp.~a safe sink) $s$ in $S$. 
Furthermore, such a vertex $s$ can be found in polynomial time. 
\end{lemma}

\begin{proof}
Let $S$ be an inclusionwise minimal vertex set in $\mathcal{F}_{\rm in}$. 
If $S=V$, then $s = r$ satisfies the conditions. Hence, it suffices to consider the case of $S \subseteq V - r$.  
In this case, we obtain $\delta^+_D(S) = \delta_G(S) - \delta^-_D(S) \ge (2k+2) - k = k+2$. 
Let
\[
\mathcal{F}^S_{\rm out} := \{ X \in \mathcal{F}_{\rm out} \mid X \subsetneq S \} = \{ X \subseteq S \mid \delta^+_D(X) = k \}
\]
and let $Y_1, Y_2, \dots , Y_{\alpha-1}$, and $Y_\alpha$ be the inclusionwise maximal vertex sets in $\mathcal{F}^S_{\rm out}$. 
Note that these sets are mutually disjoint, because $Y_i \cap Y_j \not= \emptyset$ implies $Y_i \cup Y_j \in \mathcal{F}^S_{\rm out}$ by Lemma~\ref{lem:102}. 
Let  
\[
\mathcal{G} := \left\{ Z \subseteq S - \bigcup_{i=1}^\alpha Y_i  \;\middle|\; \delta^+_D (Z) = k +1  \right\} 
\]
and let $Z_1, Z_2, \dots , Z_{\beta-1}$, and $Z_\beta$ be the inclusionwise maximal vertex sets in $\mathcal{G}$ (see Figure~\ref{fig:04}). 
We show the following two claims. 

\begin{figure}[t]
\centering
\includegraphics[width=5cm]{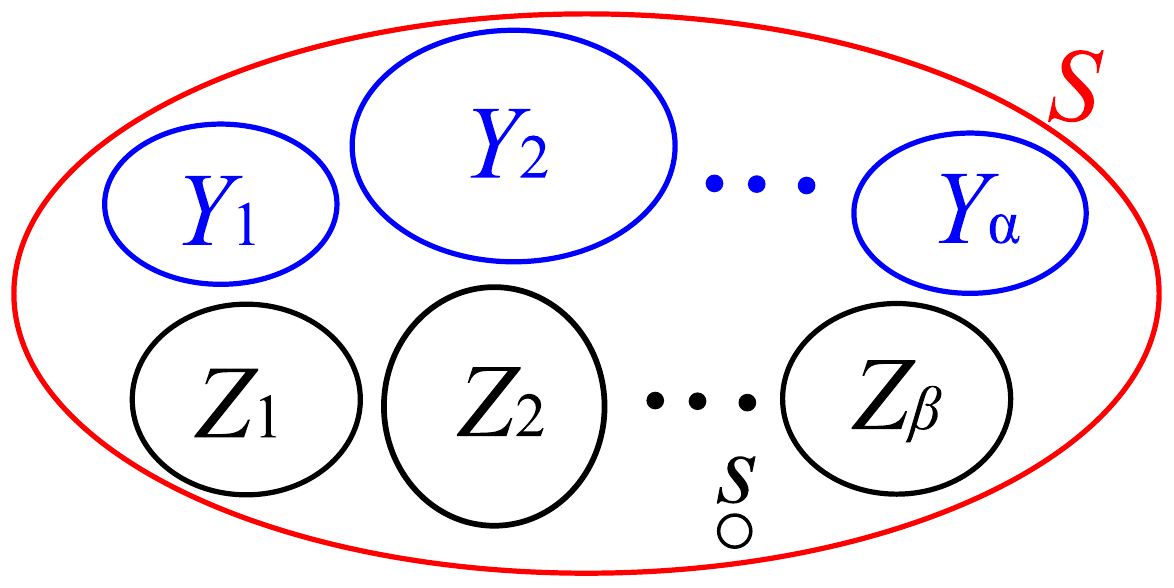}
\caption{$Y_i$, $Z_j$, and $s$.}
\label{fig:04}
\end{figure}

\begin{claim}
\label{clm:mutdisj}
$Z_1, Z_2, \dots , Z_{\beta-1}$, and $Z_\beta$ are mutually disjoint. 
\end{claim}
\begin{proof}[Proof of Claim \ref{clm:mutdisj}]
Assume to the contrary that $Z_i \cap Z_j \not= \emptyset$ for some distinct $i, j \in \{1, \dots , \beta\}$. 
Since $Z_i \cap Z_j \not\in \mathcal{F}^S_{\rm out}$ and $Z_i \cup Z_j \not\in \mathcal{F}^S_{\rm out}$, 
we obtain $\delta^+_D (Z_i \cap Z_j) \ge k+1$ and $\delta^+_D (Z_i \cup Z_j) \ge k+1$. 
Then, it holds that 
\begin{linenomath}
\begin{align*}
    2(k+1) = \delta^+_D (Z_i) + \delta^+_D (Z_j) \ge \delta^+_D (Z_i \cap Z_j) + \delta^+_D (Z_i \cup Z_j) \ge 2 (k+1).  
\end{align*}
\end{linenomath}
Therefore, $\delta^+_D (Z_i \cup Z_j) = k+1$, and hence $Z_i \cup Z_j \in \mathcal{G}$. 
This contradicts the maximality of $Z_i$ and $Z_j$. 
\end{proof}

\begin{claim}
\label{clm:nonempty}
$S - \bigcup_{i=1}^\alpha Y_i - \bigcup_{j=1}^\beta Z_j \not= \emptyset$.  
\end{claim}
\begin{proof}[Proof of Claim \ref{clm:nonempty}]
Assume to the contrary that $S - \bigcup_{i=1}^\alpha Y_i - \bigcup_{j=1}^\beta Z_j = \emptyset$. 
Then, 
\begin{linenomath}
\begin{align*}
    k \alpha + (k+1) \beta - (k+2) 
      &\ge \sum_{i=1}^\alpha \delta^+_D(Y_i) + \sum_{j=1}^\beta \delta^+_D(Z_j) - \delta^+_D (S) \\ 
      &= \sum_{i=1}^\alpha \delta^-_D(Y_i) + \sum_{j=1}^\beta \delta^-_D(Z_j) - \delta^-_D (S) \\ 
      &\ge \sum_{i=1}^\alpha (2k+2-\delta^+_D(Y_i)) + \sum_{j=1}^\beta (2k+2-\delta^+_D(Z_j)) - k \\ 
      &= (k+2) \alpha + (k+1) \beta - k \\
      &>    k \alpha + (k+1) \beta - (k+2), 
\end{align*}
\end{linenomath}
which is a contradiction. 
\end{proof}

By Claim \ref{clm:nonempty}, we can choose a vertex $s \in S - \bigcup_{i=1}^\alpha Y_i - \bigcup_{j=1}^\beta Z_j$ (see Figure~\ref{fig:04}). 
We now show that $s$ is a safe source in $S$. 
Recall that a vertex $s \in S$ is called a safe source in $S$ if, for any $X \subseteq V - r$ with $s \in X$ and $S - X \not= \emptyset$, 
\begin{enumerate}
    \item $\delta^+_D (X) \ge k+1$ holds, and 
    \item if $\delta^+_D (X) = k+1$, then there exists a vertex set $X' \subseteq X -s$ with $X' \in \mathcal{F}_{\rm out}$.   
\end{enumerate}

To show the first condition, assume to the contrary that $\delta^+_D (X) = k$ holds for some $X \subseteq V - r$ with $s \in X$ and $S - X \not= \emptyset$. 
Since $s \in X$ implies that $X \not\in \mathcal{F}^S_{\rm out}$, $X$ is not a subset of $S$, i.e., $X - S \not= \emptyset$. 
Then, $S - X \in \mathcal{F}_{\rm in}$ by Lemma~\ref{lem:103}, which contradicts the minimality of $S$ as $S-X \subseteq S-s \subsetneq S$. 
Therefore, the first condition is satisfied. 

To show the second condition, suppose that $\delta^+_D (X) = k+1$ holds for some $X \subseteq V - r$ with $s \in X$ and $S - X \not= \emptyset$.  
We treat the case of $X \subseteq S$ and that of $X-S \not= \emptyset$, separately.  
\begin{itemize}
    \item Suppose that $X \subseteq S$. 
    Since $X$ and its supersets are not in $\mathcal{G}$ by the choice of $s$, $X$ is not contained in $S - \bigcup_{i=1}^\alpha Y_i$, that is, 
    $X \cap Y_i \not= \emptyset$ for some $i \in \{1, \dots , \alpha\}$. 
    By the $k$-edge-connectedness of $D$, it holds that $\delta^+_D (X \cap Y_i) \ge k$. 
    Since $X \cup Y_i \supseteq Y_i + s$, we obtain $X \cup Y_i \not\in \mathcal{F}^S_{\rm out}$ by the maximality of $Y_i$, which implies that $\delta^+_D (X \cup Y_i) \ge k+1$. 
    Then, we obtain 
    \[
        2k+1 = \delta^+_D (X) + \delta^+_D (Y_i) \ge \delta^+_D (X \cap Y_i) + \delta^+_D (X \cup Y_i) \ge 2k+1, 
    \]
    and hence $\delta^+_D (X \cap Y_i) = k$ and $\delta^+_D (X \cup Y_i) = k+1$. 
    Therefore, $X' := X \cap Y_i$ satisfies that $X' \subseteq X - s$ and $X' \in \mathcal{F}_{\rm out}$ (see Figure~\ref{fig:05}).

    \item
    Suppose that $X-S \not= \emptyset$. 
    By the $k$-edge-connectedness of $D$, it holds that $\delta^+_D (X - S) \ge k$. 
    Since $S-X \subseteq S -s \subsetneq S$, we obtain $S-X \not\in \mathcal{F}_{\rm in}$ by the minimality of $S$, which implies that $\delta^-_D (S-X) \ge k+1$. 
    Then, by Lemma~\ref{lem:101}, we obtain 
    \[
        2k+1 = \delta^+_D (X) + \delta^-_D (S) \ge \delta^+_D (X - S) + \delta^-_D (S - X) \ge 2k+1, 
    \]
    and hence $\delta^+_D (X-S) = k$ and $\delta^-_D (S-X) = k+1$. 
    Therefore, $X' := X-S$ satisfies that $X' \subseteq X - s$ and $X' \in \mathcal{F}_{\rm out}$ (see Figure~\ref{fig:06}). 
\end{itemize}

By this argument, $s$ is a safe source in $S$. Furthermore, since 
$Y_1, Y_2, \dots , Y_{\alpha}$, $Z_1, Z_2, \dots , Z_{\beta-1}$, and $Z_\beta$ 
can be computed by using a minimum cut algorithm, 
a vertex $s \in S - \bigcup_{i=1}^\alpha Y_i - \bigcup_{j=1}^\beta Z_j$ can be found in polynomial time. 

By the same argument, if $S$ is an inclusionwise minimal vertex set in $\mathcal{F}_{\rm out}$, 
then a safe sink $s$ in $S$ can be found in polynomial time. 
\end{proof}

\begin{figure}[t]
 \begin{minipage}{0.59\hsize}
    \centering
    \includegraphics[width=6cm]{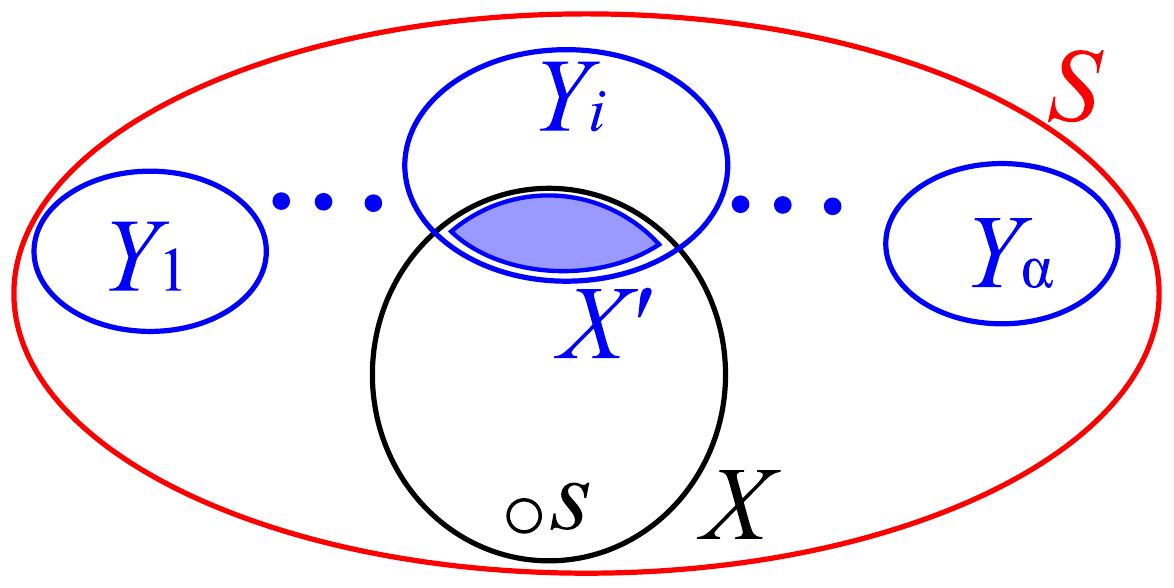}
    \caption{Case of $X \subseteq S$.}
    \label{fig:05}
 \end{minipage}
 \begin{minipage}{0.39\hsize}
    \centering
    \includegraphics[width=4.5cm]{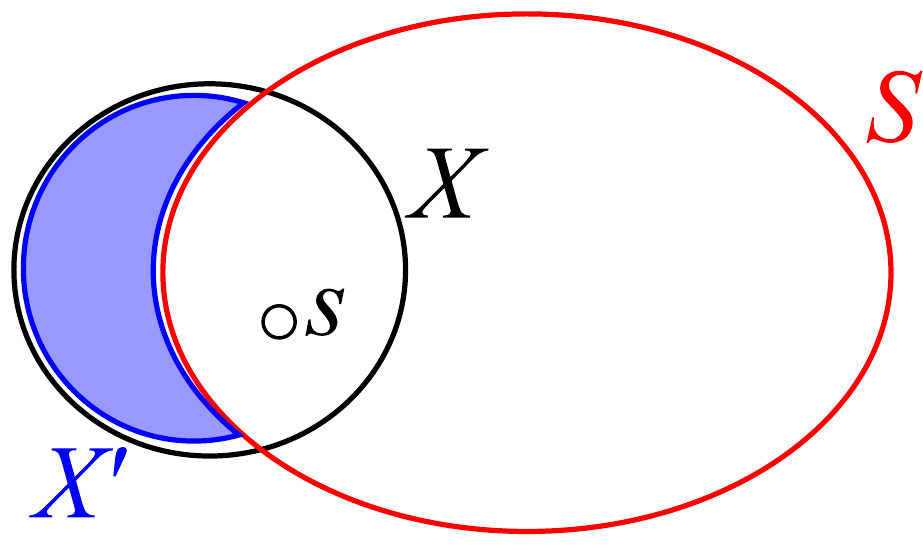}
    \caption{Case of $X-S \not= \emptyset$.}
    \label{fig:06}
 \end{minipage}
\end{figure}

\subsubsection{Path to a Minimal Vertex Set}
\label{sec:existpath}

The goal of this sub-subsection is to show Lemma~\ref{lem:105} below, saying that, for any vertex $s\in V$, $D$ has a path $P$ from $s$ to some inclusionwise minimal set $T$ in $\mathcal{F}_{\rm out}$ such that $P$ leaves no set in $\mathcal{F}_{\rm out}$.
Analogously to Lemma~\ref{lem:105}, we can obtain a path to any vertex $t\in V$ from some inclusionwise minimal set $S$ in $\mathcal{F}_{\rm in}$~(Lemma~\ref{lem:105'}). 
These paths will be used in our proof of Lemma~\ref{lem:defP}. 

\begin{lemma}
\label{lem:105}
For any vertex $s \in V$, there exists a vertex set $T \in \mathcal{F}_{\rm out}$
satisfying the following conditions: 
\begin{itemize}
    \item $T$ is inclusionwise minimal in $\mathcal{F}_{\rm out}$, and 
    \item for any vertex $t \in T$, $D$ contains an $(s,t)$-path $P_t$ such that $A(P_t) \cap \Delta^+_D(X) = \emptyset$ for any $X \in \mathcal{F}_{\rm out}$. 
\end{itemize}
Furthermore, such $T$ and $P_t$ can be found in polynomial time. 
\end{lemma}

To prove the lemma, we need more definitions.
For a vertex $s \in V$, let $X_{\rm out}(s)$ denote the inclusionwise minimal vertex set subject to $s \in X_{\rm out}(s) \in \mathcal{F}_{\rm out}$. 
Note that such a vertex set always exists as $s \in V \in \mathcal{F}_{\rm out}$. 
Note also that the minimal one is uniquely determined, 
because if $s \in X \in \mathcal{F}_{\rm out}$ and $s \in Y \in \mathcal{F}_{\rm out}$, then $s \in X \cap Y \in \mathcal{F}_{\rm out}$  by Lemma~\ref{lem:102}. 
For each $s \in V$, we can easily compute $X_{\rm out}(s)$ in polynomial time by using a minimum cut algorithm.

\begin{lemma}
\label{lem:104}
Let $s \in V$. For any vertex $t \in X_{\rm out}(s)$, 
$D[X_{\rm out}(s)]$ contains a path from $s$ to $t$. 
\end{lemma}

\begin{proof}
If $X_{\rm out}(s) = V$, then the lemma holds since $D$ is strongly connected, where we note that $k\ge 1$.
Thus, we consider the case when $X_{\rm out}(s) \neq V$.
Assume to the contrary that $D[X_{\rm out}(s)]$ does not contain a path from $s$ to $t$. 
Then, there exists a vertex set $S \subseteq X_{\rm out}(s)$ such that $s \in S$, $t \in X_{\rm out}(s) - S$, and 
$D$ has no edge from $S$ to $X_{\rm out}(s) - S$ (Figure \ref{fig:01}). 
Since $X_{\rm out}(s) \in \mathcal{F}_{\rm out}$ and $D$ is $k$-edge-connected, we obtain 
\[
k = \delta^+_D(X_{\rm out}(s)) \ge \delta^+_D(S) \ge k, 
\]
and hence $S \in \mathcal{F}_{\rm out}$. 
Since $s \in S \subsetneq X_{\rm out}(s)$, this contradicts the minimality of $X_{\rm out}(s)$. 
\end{proof}

\begin{figure}[t]
\centering
\includegraphics[width=5cm]{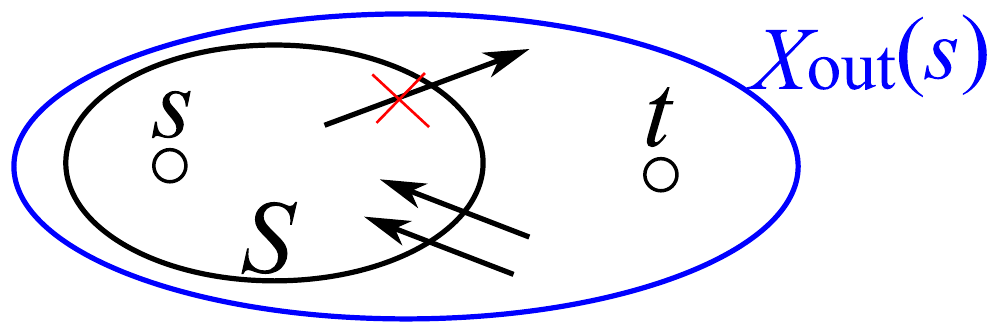}
\caption{Proof of Lemma~\ref{lem:104}.}
\label{fig:01}
\end{figure}

We are ready to prove Lemma~\ref{lem:105}.

\begin{proof}[Proof of Lemma~\ref{lem:105}]
We prove the lemma by induction on $|X_{\rm out}(s)|$. 

If $X_{\rm out}(s)$ is an inclusionwise minimal vertex set in $\mathcal{F}_{\rm out}$, 
then $T := X_{\rm out}(s)$ satisfies the condition. 
This is because the existence of $P_t$ is guaranteed by Lemma~\ref{lem:104} and 
$A(P_t) \cap \Delta^+_D(X) = \emptyset$ for any $X \in \mathcal{F}_{\rm out}$ follows from 
the minimality of $T=X_{\rm out}(s)$ and Lemma~\ref{lem:102}. 
This is the base case of the induction. 

Suppose that $X_{\rm out}(s)$ is not an inclusionwise minimal vertex set in $\mathcal{F}_{\rm out}$, that is, 
there exists a vertex set $Y$ in $\mathcal{F}_{\rm out}$ that is strictly contained in $X_{\rm out}(s)$.  
We can take a vertex $u \in X_{\rm out}(s)$ such that $X_{\rm out}(u) = Y$. 
By Lemma~\ref{lem:104}, $D[X_{\rm out}(s)]$ contains a path $Q$ from $s$ to $u$. 
Traverse along $Q$ from $s$ to $u$ and let $s'$ be the first vertex on $Q$ such that $X_{\rm out}(s') \subsetneq X_{\rm out}(s)$ (see Figure~\ref{fig:02}). 
Note that such $s'$ always exists as $u$ satisfies the condition.
Recall that $Q[s, s']$ denotes the subpath of $Q$ between $s$ and $s'$. 
We show the following claim. 

\begin{claim}
\label{clm:subpath}
$A(Q[s, s']) \cap \Delta^+_D(X) = \emptyset$ for any $X \in \mathcal{F}_{\rm out}$. 
\end{claim}

\begin{proof}[Proof of Claim \ref{clm:subpath}]
Assume to the contrary that there exists an edge $(x,y) \in A(Q[s, s']) \cap \Delta^+_D(X^*)$ for some $X^* \in \mathcal{F}_{\rm out}$. 
Then $X^* \cap X_{\rm out}(s) \in \mathcal{F}_{\rm out}$ by Lemma~\ref{lem:102}.
Since $x \in X^* \cap X_{\rm out}(s)$, we see that $X_{\rm out}(x)\subseteq X^* \cap X_{\rm out}(s)$.
Therefore, since $X^* \cap X_{\rm out}(s) \subseteq X_{\rm out}(s) - y \subsetneq X_{\rm out}(s)$, we obtain $X_{\rm out}(x) \subsetneq X_{\rm out}(s)$, which contradicts the choice of $s'$. 
Thus, Claim \ref{clm:subpath} follows.
\end{proof}

Since $|X_{\rm out}(s')| < |X_{\rm out}(s)|$, by the induction hypothesis, 
there exists an inclusionwise minimal vertex set $T$ of $\mathcal{F}_{\rm out}$
satisfying the following condition: 
for any vertex $t \in T$, 
$D$ contains an $(s', t)$-path $P'_t$ such that 
$A(P'_t) \cap \Delta^+_D(X) = \emptyset$ for any $X \in \mathcal{F}_{\rm out}$. 
We remark that $P'_t$ is contained in $X_{\rm out}(s')$ as it contains no edge in $\Delta^+_D(X_{\rm out}(s'))$, noting that $X_{\rm out}(s') \in \mathcal{F}_{\rm out}$.

We now show that $T$ is a desired set also for $s$. 
For each $t \in T$, 
let $P_t$ be the $(s, t)$-path obtained by concatenating $Q[s, s']$ and $P'_t$. 
Note that $P_t$ is indeed a path (i.e., it goes through each vertex at most once), 
since $V(Q[s, s']) \cap X_{\rm out}(s') = \{s'\}$ and $V(P'_t) \subseteq X_{\rm out}(s')$. 
Then, Claim~\ref{clm:subpath} shows that 
$A(P_t) \cap \Delta^+_D(X) = (A(Q[s, s']) \cap \Delta^+_D(X)) \cup (A(P'_t) \cap \Delta^+_D(X)) = \emptyset$ for any $X \in \mathcal{F}_{\rm out}$.  
Hence $T$ is a desired set for $s$. 

\begin{figure}[t]
\centering
\includegraphics[width=8cm]{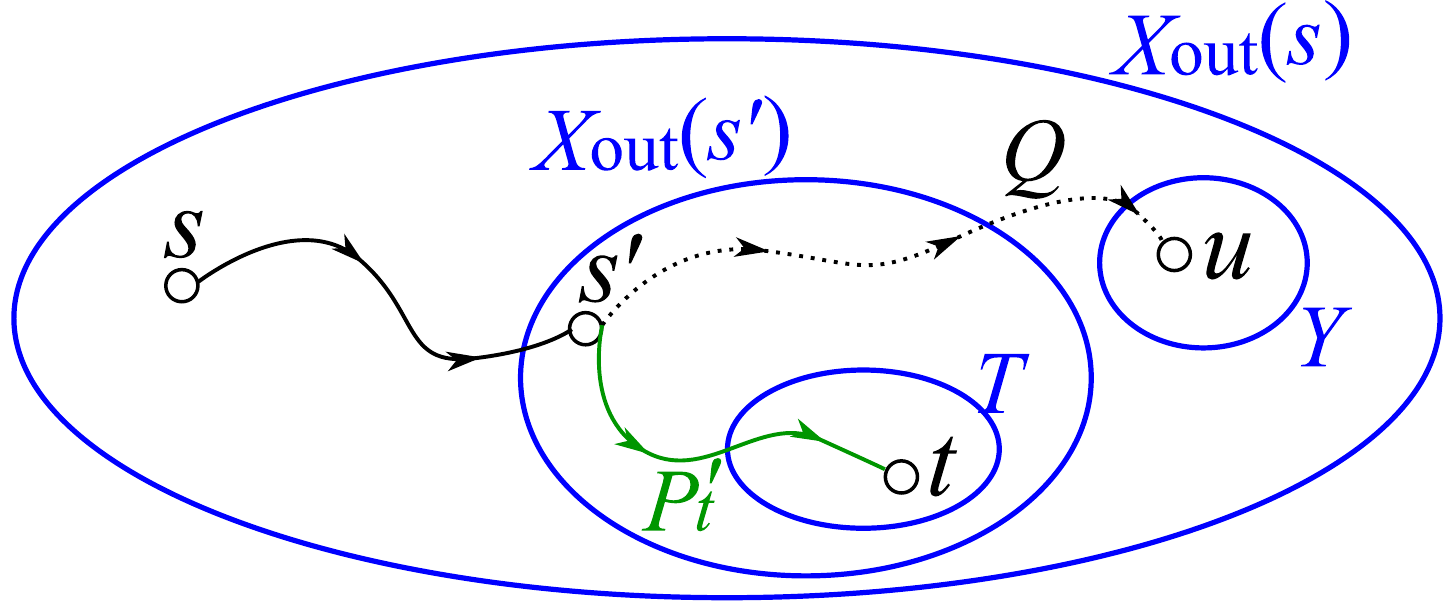}
\caption{Proof of Lemma~\ref{lem:105}.}
\label{fig:02}
\end{figure}

Since the above inductive proof can be converted to a recursive algorithm, 
$T$ and $P_t$ satisfying the conditions can be computed in polynomial time. 
\end{proof}

We note that the vertex set $T$ in Lemma~\ref{lem:105} is not necessarily an element of $\mathcal{F}_{\rm min}$, 
because there might exist a vertex set $U \subsetneq T$ with $U \in \mathcal{F}_{\rm in}$. 
By changing the roles of $\mathcal{F}_{\rm out}$ and $\mathcal{F}_{\rm in}$, we obtain the following lemma in the same way as Lemma~\ref{lem:105}. 

\begin{lemma}
\label{lem:105'}
For any vertex $t \in V$, there exists a vertex set $S \in \mathcal{F}_{\rm in}$
satisfying the following conditions: 
\begin{itemize}
    \item $S$ is inclusionwise minimal in $\mathcal{F}_{\rm in}$, and 
    \item for any vertex $s \in S$, $D$ contains an $(s,t)$-path $P_s$ such that $A(P_s) \cap \Delta^-_D(X) = \emptyset$ for any $X \in \mathcal{F}_{\rm in}$. 
\end{itemize}
Furthermore, such $S$ and $P_s$ can be found in polynomial time. 
\qed
\end{lemma}

\subsubsection{Proof of Lemma~\ref{lem:defP}}
\label{sec:subsectiondefP}

We are now ready to prove Lemma~\ref{lem:defP}. 

Let $T^* \in \mathcal{F}^R_{\rm out}$ and $t^* \in T^*$. 
By applying Lemma~\ref{lem:105'} to $t^*$, 
we obtain an inclusionwise minimal vertex set $S$ in $\mathcal{F}_{\rm in}$.
By Lemma~\ref{lem:107}, $S$ has a safe source $s$.
Lemma~\ref{lem:105'} guarantees that there exists an $(s,t^*)$-path $P_s$ such that 
$A(P_s) \cap \Delta^-_D(X) = \emptyset$ for any $X \in \mathcal{F}_{\rm in}$.
This implies that $P_s$ is in $D[R]$, since $A(P_s) \cap \Delta^-_D(R) = \emptyset$ as $R \in \mathcal{F}_{\rm in}$.
In particular, $s \in R$. 
By the minimality of $S$, this shows that $S \subseteq R$ (possibly, $S=R$). 

Traverse along $P_s$ from $s$ to $t^*$ and let $t'$ be the first vertex on $P_s$ such that 
there exists a vertex set $T' \in \mathcal{F}^R_{\rm out}$ with $t' \in T'$. 
Note that such $t'$ always exists, because $t'=t^*$ satisfies the condition. 
We also note that, for each $x \in V$, we can check the existence of a vertex set $X \in \mathcal{F}^R_{\rm out}$ with $x \in X$ by a minimum cut algorithm. 
Let $Q_1$ denote the subpath of $P_s$ between $s$ and $t'$, i.e., $Q_1 := P_s[s, t']$. 
Then we see by the choice of $t'$ that $V(Q_1)-t'$ is disjoint from $X$ for every $X\in  \mathcal{F}^R_{\rm out}$.

By applying Lemma~\ref{lem:105} to $t'$, 
we obtain an inclusionwise minimal vertex set $T$ in $\mathcal{F}_{\rm out}$. 
Let $t$ be a safe sink in $T$ as in Lemma~\ref{lem:107}, and  
let $Q_2$ be a $(t', t)$-path such that 
$A(Q_2) \cap \Delta^+_D(X) = \emptyset$ for any $X \in \mathcal{F}_{\rm out}$, 
whose existence is guaranteed by Lemma~\ref{lem:105}. 
We note that $V(Q_2) \subseteq T' \subsetneq R$, 
since $A(Q_2) \cap \Delta^+_D(T') = \emptyset$ as $T' \in \mathcal{F}_{\rm out}$.
In particular, $t \in T' \subsetneq R$. 
By the minimality of $T$, this shows that $T \subseteq T' \subsetneq R$ (possibly, $T=T'$). 

Let $P$ be the $(s, t)$-path obtained by concatenating $Q_1$ and $Q_2$ (see Figure~\ref{fig:07}). 
Note that $P$ goes through each vertex at most once, since 
$V(Q_1) \cap T' = \{t'\}$ and $V(Q_2) \subseteq T'$. 
Then, $P$ satisfies the conditions in the lemma. 
By Lemmas~\ref{lem:107}, \ref{lem:105}, and \ref{lem:105'}, the above procedure can be executed in polynomial time, which completes the proof. 
\qed

\begin{figure}[t]
\centering
\includegraphics[width=7cm]{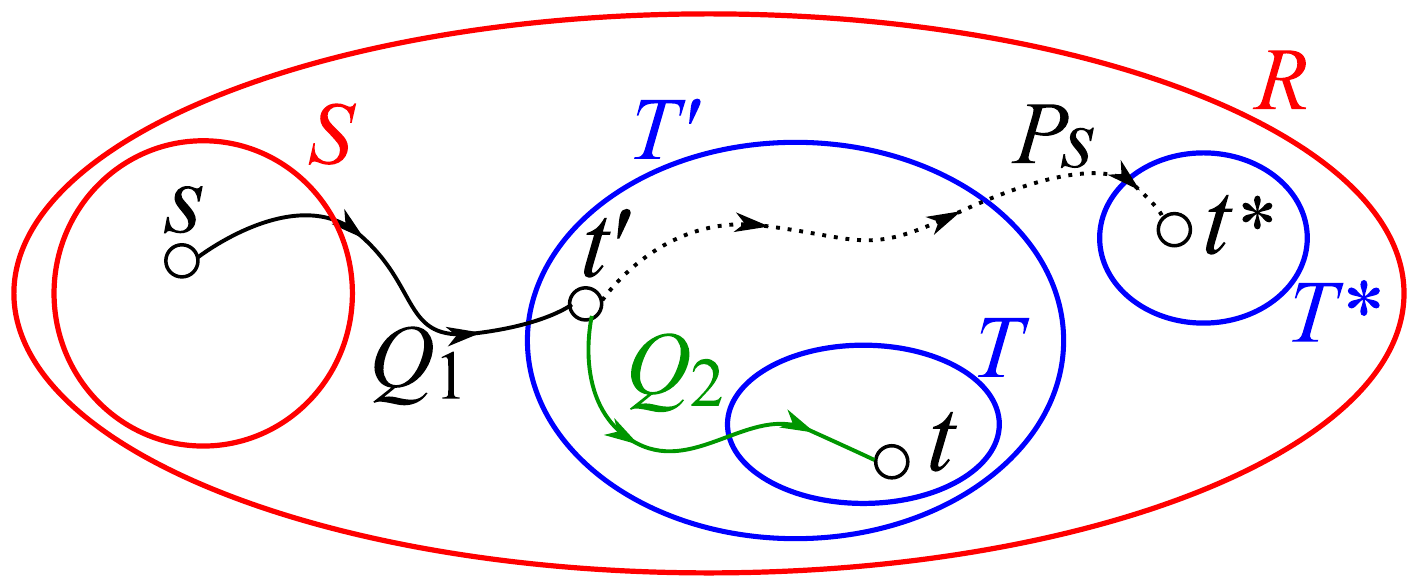}
\caption{Proof of Lemma~\ref{lem:defP}.}
\label{fig:07}
\end{figure}

\section{Proof of Theorem~\ref{thm:local}}
\label{sec:localproof}

In this section, we prove Theorem~\ref{thm:local}: 
for two strongly connected orientations $D_1, D_2$,  
there exists a path connecting $D_1$ and $D_2$ in the edge-flip graph $\GG_1(G)$ if and only if there exists no $2$-edge-cut $\{\{u,v\}, \{u',v'\}\}$ such that
$(u,v), (v',u')$ are edges of $D_1$ and
$(v,u), (u',v')$ are edges of $D_2$.

If there exists a $2$-edge-cut $\{\{u,v\}, \{u',v'\}\}$ such that
$(u,v), (v',u')$ are edges of $D_1$ and
$(v,u), (u',v')$ are edges of $D_2$, then
we cannot flip $(u, v)$ or $(v', u')$ in $D_1$ one-by-one.
Hence $D_1$ and $D_2$ are not connected in the edge-flip graph $\GG_1(G)$.

Suppose that no such $2$-edge-cut exists.
Let $D_1 = (V, A_1)$ and $D_2=(V, A_2)$.
We will show that there exists an edge $e$ in $A_1 -  A_2$ such that we can flip $e$ in $D_1$ keeping strong edge-connectedness.

Let $e=(u, v)$ be an arbitrary edge in $D_1$ but not in $D_2$~(i.e., $(v, u)$ is in $D_2$).
We may assume that we cannot flip $e$.
Then there exists a vertex subset $X$ such that $\Delta^+_{D_1} (X)= \{e\}$.
Since $D_1$ is strongly connected, $\delta^-_{D_1}(X)\geq 1$.
We take such a vertex subset $X$ so that $\delta^-_{D_1}(X)$ is maximized.

Since $D_2$ has an edge $(v, u)$ and $\delta^+_{D_2}(X)\geq 1$, there exists an edge $f$
in $\Delta^-_{D_1}(X)$ such that 
the reverse of $f$ is in $A_2$.
We will show that the edge $f$ can be flipped keeping strong connectivity.

Assume to the contrary that $f$ cannot be flipped.
Then there exists a vertex subset $Y$ such that $\Delta^+_{D_1} (Y) = \{f\}$ and $\delta^-_{D_1}(Y)\geq 1$.
Moreover, we see that $\delta^-_{D_1}(Y)\geq 2$.
In fact, if $\delta^-_{D_1}(Y) = 1$, then $E_G(Y)$ is a $2$-edge cut.
Since $D_2$ is strongly connected, the (unique) edge in $\Delta^-_{D_1}(Y)$ must be flipped in $D_2$.
This, however, contradicts the assumption of the theorem.

First consider the case when $X\cap Y = \emptyset$.
If $X\cup Y = V$, then $\Delta^+_{D_1} (X) = \{e\}$ and $\Delta^-_{D_1} (X)=\Delta^+_{D_1} (Y) = \{f\}$, which contradicts the assumption of the theorem.
Thus we have $X\cup Y \subsetneq V$.
Define $X' = X\cup Y$.
Then, since the edge $f$ only enters from $Y$ to $X$, it follows that $\delta^-_{D_1} (X') = \delta^-_{D_1} (X) -1 + \delta^-_{D_1}(Y)$.
Since $\delta^-_{D_1}(Y)\geq 2$, we have $\delta^-_{D_1} (X') > \delta^-_{D_1} (X)$, which contradicts the maximality of $\delta^-_{D_1}(X)$.

Thus we may suppose that $X\cap Y \neq \emptyset$.
We first claim that $X\cup Y = V$.
Indeed, if $X\cup Y \subsetneq V$, then we have
\[
2 = \delta^+_{D_1} (X) +\delta^+_{D_1} (Y)\geq \delta^+_{D_1} (X\cap Y) +\delta^+_{D_1} (X\cup Y) \geq 2
\]
and hence $\delta^+_{D_1} (X\cap Y) = \delta^+_{D_1} (X\cup Y) = 1$.
However, we see that $\Delta^+_{D_1}(X\cap Y) \cup \Delta^+_{D_1}(X\cup Y)= \{e\}$
and $e \notin \Delta^+_{D_1}(X\cap Y) \cap \Delta^+_{D_1}(X\cup Y)$, which is a contradiction.

Define $X' = X\cap Y$.
Since the edge $f$ only enters from $Y -  X$ to $X -  Y$, it follows that $\delta^-_{D_1} (X') = \delta^-_{D_1} (X) -1 + \delta^-_{D_1}(Y)$.
Since $\delta^-_{D_1}(Y)\geq 2$, we have $\delta^-_{D_1} (X') > \delta^-_{D_1} (X)$, which contradicts the maximality of $\delta^-_{D_1}(X)$.
Therefore, the edge $f$ can be flipped.

From the above discussion, we can find in polynomial time an edge $e$ in $A_1 -  A_2$ such that flipping the edge $e$ in $D_1$ does not violate strong edge-connectedness.
By repeatedly finding such edges, we obtain a sequence of orientations from $D_1$ to $D_2$ by edge flips. 
In each edge flip, $|A_1 - A_2|$ decreases by one.
Since the length of a sequence is at least $|A_1 - A_2|$, the obtained sequence turns out to be the shortest.
This completes the proof of Theorem~\ref{thm:local}.
\qed

\section{Concluding Remarks}
\label{sec:conclusion}

This paper initiates the study of $k$-edge-connected orientations through edge flips for $k \geq 2$.
As a showcase, we give a new edge-flip-based proof (Theorem \ref{thm:main01}) of Nash-Williams' theorem \cite{nash-williams_1960}: an undirected graph $G$ has a $k$-edge-connected orientation if and only if $G$ is $2k$-edge-connected.
Our new proof has another useful property that all the intermediate orientations have non-decreasing edge-connectivity in the process.
Using Theorem \ref{thm:main01}, we prove that the edge-flip graph of $k$-edge-connected orientations of an undirected graph $G$ is connected if $G$ is $(2k+2)$-edge-connected (Theorem \ref{thm:main}).

Several questions remain open.
In Theorem \ref{thm:main01}, we showed that the length of an edge-flip sequence is bounded by $k|V|^3$.
However, we do not know this bound is tight.
It is not clear how to find such a shortest sequence in polynomial time.
We do not know the tightness of Theorem \ref{thm:main}, either:
we do not know whether the edge-flip graph of $k$-edge-connected orientations is connected when the underlying undirected graph is $(2k+1)$-edge-connected.
We do not know the $k$-edge-connectedness counterpart of Theorem \ref{thm:local} when $k\geq 2$.

It is not clear how to find a shortest path in the edge-flip graph of $k$-edge-connected orientations in polynomial time when $k\geq 2$.
When $k=1$, we can find a shortest path by looking at the ``symmetric difference'' of two given strongly connected orientations~\cite{GZ,FPS}.
However, when $k=2$, there exists an example for which the symmetric difference does not determine a shortest path.

\section*{Acknowledgments}

This work was supported by 
JSPS KAKENHI Grant Numbers JP20H05793, JP20H05795, JP20K11670, JP20K11692, JP19K11814, JP18H04091, JP18H05291, and JP21H03397, Japan.

\bibliographystyle{acm}
\bibliography{orient}

\begin{thebibliography}{10}

\bibitem{AichholzerCHKMS21}
{\sc Aichholzer, O., Cardinal, J., Huynh, T., Knauer, K., M{\"{u}}tze, T.,
  Steiner, R., and Vogtenhuber, B.}
\newblock Flip distances between graph orientations.
\newblock {\em Algorithmica 83}, 1 (2021), 116--143.

\bibitem{DBLP:journals/jgt/BergJ06}
{\sc Berg, A.~R., and Jord{\'{a}}n, T.}
\newblock Two-connected orientations of {E}ulerian graphs.
\newblock {\em Journal of Graph Theory 52}, 3 (2006), 230--242.

\bibitem{BERNATH2008663}
{\sc Bern\'{a}th, A., Iwata, S., Kir\'{a}ly, T., Kir\'{a}ly, Z., and Szigeti,
  Z.}
\newblock Recent results on well-balanced orientations.
\newblock {\em Discrete Optimization 5}, 4 (2008), 663--676.

\bibitem{CHERIYAN201417}
{\sc Cheriyan, J., {Durand de Gevigney}, O., and Szigeti, Z.}
\newblock Packing of rigid spanning subgraphs and spanning trees.
\newblock {\em Journal of Combinatorial Theory, Series B 105\/} (2014), 17--25.

\bibitem{DURANDDEGEVIGNEY2020105}
{\sc {Durand de Gevigney}, O.}
\newblock On {F}rank's conjecture on {$k$}-connected orientations.
\newblock {\em Journal of Combinatorial Theory, Series B 141\/} (2020),
  105--114.

\bibitem{FRANK198297}
{\sc Frank, A.}
\newblock An algorithm for submodular functions on graphs.
\newblock In {\em Bonn Workshop on Combinatorial Optimization}, A.~Bachem,
  M.~Gr\"{o}tschel, and B.~Korte, Eds., vol.~66 of {\em North-Holland
  Mathematics Studies}. North-Holland, 1982, pp.~97--120.

\bibitem{Frank82}
{\sc Frank, A.}
\newblock A note on $k$-strongly connected orientations of an undirected graph.
\newblock {\em Discret. Math. 39}, 1 (1982), 103--104.

\bibitem{frank_1993}
{\sc Frank, A.}
\newblock Applications of submodular functions.
\newblock In {\em Surveys in Combinatorics, 1993}, K.~Walker, Ed., London
  Mathematical Society Lecture Note Series. Cambridge University Press, 1993,
  pp.~85--136.

\bibitem{10.5555/233157.233167}
{\sc Frank, A.}
\newblock Connectivity and network flows.
\newblock In {\em Handbook of Combinatorics}, vol.~1. MIT Press, 1996,
  p.~111–177.

\bibitem{frank_book}
{\sc Frank, A.}
\newblock {\em Connections in Combinatorial Optimization}.
\newblock Oxford University Press, 2011.

\bibitem{FRANK2003385}
{\sc Frank, A., Kir\'aly, T., and Kir\'aly, Z.}
\newblock On the orientation of graphs and hypergraphs.
\newblock {\em Discrete Applied Mathematics 131}, 2 (2003), 385--400.

\bibitem{FPS}
{\sc Fukuda, K., Prodon, A., and Sakuma, T.}
\newblock Notes on acyclic orientations and the shelling lemma.
\newblock {\em Theor. Comput. Sci. 263}, 1-2 (2001), 9--16.

\bibitem{FUKUNAGA20122349}
{\sc Fukunaga, T.}
\newblock Graph orientations with set connectivity requirements.
\newblock {\em Discrete Mathematics 312}, 15 (2012), 2349--2355.

\bibitem{Gabow94}
{\sc Gabow, H.~N.}
\newblock Efficient splitting off algorithms for graphs.
\newblock In {\em Proceedings of the Twenty-Sixth Annual {ACM} Symposium on
  Theory of Computing, 23-25 May 1994, Montr{\'{e}}al, Qu{\'{e}}bec, Canada\/}
  (1994), F.~T. Leighton and M.~T. Goodrich, Eds., {ACM}, pp.~696--705.

\bibitem{GZ}
{\sc Greene, C., and Zaslavsky, T.}
\newblock On the interpretation of {W}hitney numbers through arrangements of
  hyperplanes, zonotopes, non-{R}adon partitions, and orientations of graphs.
\newblock {\em Trans. Amer. Math. Soc. 280\/} (1983), 97--126.

\bibitem{hakimi}
{\sc Hakimi, S.~L.}
\newblock On the degrees of the vertices of a directed graph.
\newblock {\em J. Franklin Inst. 279}, 4 (1965), 290--308.

\bibitem{HausknechtASFW11}
{\sc Hausknecht, M.~J., Au, T., Stone, P., Fajardo, D., and Waller, S.~T.}
\newblock Dynamic lane reversal in traffic management.
\newblock In {\em 14th International {IEEE} Conference on Intelligent
  Transportation Systems, {ITSC} 2011, Washington, DC, USA, October 5-7,
  2011\/} (2011), {IEEE}, pp.~1929--1934.

\bibitem{HopcroftT73}
{\sc Hopcroft, J.~E., and Tarjan, R.~E.}
\newblock Efficient algorithms for graph manipulation (algorithm 447).
\newblock {\em Commun. {ACM} 16}, 6 (1973), 372--378.

\bibitem{ItoKK0O19}
{\sc Ito, T., Kakimura, N., Kamiyama, N., Kobayashi, Y., and Okamoto, Y.}
\newblock Shortest reconfiguration of perfect matchings via alternating cycles.
\newblock In {\em 27th Annual European Symposium on Algorithms, {ESA} 2019,
  September 9-11, 2019, Munich/Garching, Germany\/} (2019), M.~A. Bender,
  O.~Svensson, and G.~Herman, Eds., vol.~144 of {\em LIPIcs}, Schloss Dagstuhl
  - Leibniz-Zentrum f{\"{u}}r Informatik, pp.~61:1--61:15.

\bibitem{IwataK10}
{\sc Iwata, S., and Kobayashi, Y.}
\newblock An algorithm for minimum cost arc-connectivity orientations.
\newblock {\em Algorithmica 56}, 4 (2010), 437--447.

\bibitem{JORDAN2005257}
{\sc Jordán, T.}
\newblock On the existence of {$k$} edge-disjoint {$2$}-connected spanning
  subgraphs.
\newblock {\em Journal of Combinatorial Theory, Series B 95}, 2 (2005),
  257--262.

\bibitem{KiralyS06}
{\sc Kir{\'{a}}ly, Z., and Szigeti, Z.}
\newblock Simultaneous well-balanced orientations of graphs.
\newblock {\em J. Comb. Theory, Ser. {B} 96}, 5 (2006), 684--692.

\bibitem{lovasz_exercises}
{\sc Lov{\'{a}}sz, L.}
\newblock {\em Combinatorial Problems and Exercises (2nd ed.)}.
\newblock North-Holland, 1993.

\bibitem{LucchesiYounger}
{\sc Lucchesi, C.~L., and Younger, D.~H.}
\newblock A minimax theorem for directed graphs.
\newblock {\em Journal of the London Mathematical Society s2-17}, 3 (1978),
  369--374.

\bibitem{NagamochiI97}
{\sc Nagamochi, H., and Ibaraki, T.}
\newblock Deterministic {$\tilde{O}$}$(nm)$ time edge-splitting in undirected
  graphs.
\newblock {\em J. Comb. Optim. 1}, 1 (1997), 5--46.

\bibitem{nash-williams_1960}
{\sc Nash-Williams, C. S. J.~A.}
\newblock On orientations, connectivity and odd-vertex-pairings in finite
  graphs.
\newblock {\em Canadian Journal of Mathematics 12\/} (1960), 555--567.

\bibitem{robbins}
{\sc Robbins, H.~E.}
\newblock A theorem on graphs, with an application to a problem of traffic
  control.
\newblock {\em The American Mathematical Monthly 46}, 5 (1939), 281--283.

\bibitem{THOMASSEN1989402}
{\sc Thomassen, C.}
\newblock Configurations in graphs of large minimum degree, connectivity, or
  chromatic number.
\newblock {\em Annals of the New York Academy of Sciences 555}, 1 (1989),
  402--412.

\bibitem{THOMASSEN201567}
{\sc Thomassen, C.}
\newblock Strongly 2-connected orientations of graphs.
\newblock {\em Journal of Combinatorial Theory, Series B 110\/} (2015), 67--78.

\end{thebibliography}

\end{document}